\documentclass[12pt]{amsart}
\numberwithin{equation}{section}

\usepackage{amssymb}
\usepackage{amsthm}
\usepackage{mathrsfs}  
\usepackage{color}
\usepackage{hyperref}

\newtheorem{thm}{Theorem}[section]
\newtheorem{prop}{Proposition}[section]
\newtheorem{lm}{Lemma}[section]
\newtheorem{cor}{Corollary}[section]
\newtheorem{rmk}{Remark}[section]

\address{Department of Mathematics, National Tsing Hua University, 101, Section 2, Kuang-Fu Road, Hsinchu 300, Taiwan, ROC}
\email{sychen.math@gmail.com}

\address{No. 151, Yingzhuan Road, Tamsui District, New Taipei City 251, Taiwan (R.O.C),  Lui-Hsien Memorial 
Science Hall.}
\email{briancheng@o365.tku.edu.tw}

\def\GL{{\rm GL}}
\def\sW{\mathscr{W}}
\def\sS{\mathscr{S}}
\def\sV{\mathscr{V}}
\def\sH{\mathscr{H}}
\def\sC{\mathscr{C}}
\def\sY{\mathscr{Y}}

\def\R{\mathbb{R}}
\def\C{\mathbb{C}}
\def\Z{\mathbb{Z}}

\def\bp{\begin{pmatrix}}
\def\ep{\end{pmatrix}}
\def\<{\langle}
\def\>{\rangle}
\def\i{{\sqrt{-1}}}

\title{Whittaker functions on $\GL_n$ via theta lifting}
\author{Shih-Yu Chen and Yao Cheng}
\date{\today}
\begin{document}
\maketitle
\begin{abstract}
In the literature, two main approaches have been used to establish explicit formulas or propagation formulas for Whittaker functions over Archimedean local fields: one based on Jacquet integrals, and the other on the analysis of systems of partial differential equations. In this paper, we introduce a third approach via explicit theta correspondence. As an example, we derive new cases of explicit formulas for Whittaker functions on $\GL_n(\C)$ and compute the associated Asai local zeta integrals.
\end{abstract}

\section{Introduction}

Explicit formulas for Whittaker functions play a crucial role in the study of automorphic forms and automorphic $L$-functions. In particular, they provide an important tool in investigations of the algebraicity of special values of $L$-functions, such as Deligne's conjecture on critical $L$-values or its automorphic analogues, as well as in the construction of $p$-adic $L$-functions. In these contexts, having sufficiently explicit formulas allows one to compute the local zeta integrals that arise naturally in the formulation of these problems. This is especially important in the Archimedean case, where the analytic complexity of Whittaker functions makes such computations substantially more delicate and difficult.


From the perspective of propagation phenomena, many explicit formulas for Whittaker functions obtained in the literature can be understood within a unified framework. In this paper, we study Whittaker functions on general linear groups over local fields via the theory of theta correspondence. More precisely, for the reductive dual pair $(\GL_n, \GL_{n+1})$ over a local field $F$ of characteristic zero, we show that an explicit theta lifting of a generic representation $\pi$ of $\GL_n(F)$ can be described by a propagation formula relating the Whittaker functions of $\theta(\pi)$ to those of $\pi$. From this point of view, several explicit formulas in the literature may be interpreted as instances of such propagation formulas. These include Shintani's formula \cite{Shintani1976} for spherical Whittaker functions over non-Archimedean local fields; the formulas of Ishii and Stade \cite{Stade1995,IS2007} for spherical Whittaker functions over Archimedean local fields; the formulas of Manabe--Ishii--Oda \cite{MIO2004} and Miyazaki \cite{Miyazaki2009} for $\GL_3(\R)$ (see also \cite{Bump1984}); the formulas of Hirano--Oda \cite{HO2009} and Hirano--Ishii--Miyazaki \cite{HIM2012} for $\GL_3(\C)$; and the formula of Ishii--Oda \cite{IO2014} for principal series representations of $\GL_n(\R)$.


One important application of such explicit formulas is the computation of Archimedean local zeta integrals. A common feature of the formulas mentioned above is that, in the Archimedean setting, they are expressed as inverse Mellin transforms of products and ratios of $\Gamma$-functions. Integral representations of this form are commonly referred to as being of ``Mellin--Barnes type'' in the literature, and they are particularly well suited for explicit computations of local zeta integrals involving Whittaker functions. For instance, Stade and Ishii--Stade \cite{Stade2002,IS2013} computed the Rankin--Selberg zeta integrals for spherical representations of $\GL_n \times \GL_m$ with $m = n, n-1,$ and $n-2$. For principal series representations of $\GL_n(\R)$, Stade \cite{Stade1999} and Ishii \cite{Ishii2018} computed the Bump--Friedberg zeta integrals. Ishii \cite{Ishii2019} computed the Pollack--Shah zeta integrals for spherical representations of $\GL_2 \times \mathrm{GSp}_4$ and $\GL_4$. Hirano--Ishii--Miyazaki \cite{HIM2022} investigated the Rankin--Selberg zeta integrals on $\GL_3 \times \GL_2$, while Ishii--Miyazaki \cite{IshiiMiyazaki2022} treated the Rankin--Selberg zeta integrals on $\GL_n \times \GL_{n-1}$. More recently, Hirano--Ishii--Miyazaki \cite{HIM2025} computed the Bump--Friedberg zeta integrals on $\GL_4(\R)$. As an application of the new explicit formulas for $\GL_n(\C)$ obtained in this paper, we compute the associated Asai local zeta integrals, which have been less explored in the literature.


One foundational approach to Whittaker functions is based on Jacquet integrals. In \cite{Jacquet1967}, Jacquet introduced Whittaker functions on Chevalley groups over local fields, together with integral representations of these functions, now known as Jacquet integrals. His results and techniques were subsequently generalized and further developed by many authors in the setting of quasi-split reductive groups. Jacquet integrals play an important role in the existence and holomorphy of Whittaker functionals and are fundamental to the Langlands--Shahidi method for local $\gamma$-factors. For surveys of these developments and further references, we refer the reader to \cite[Introduction]{Shahidi1985} and \cite[\S\,3]{Shahidi2010}. In \cite{IshiiMiyazaki2022}, the minimal $K$-type Whittaker functions for $\GL_n(\C)$ are normalized via Jacquet integrals, under which explicit computations are carried out. The resulting formulas play a crucial role in establishing the integrality of critical $L$-values for $\GL_n \times \GL_{n-1}$ over a totally imaginary number field, as shown by Hara--Miyazaki--Namikawa \cite{HMN2025}. Related propagation phenomena for Whittaker functions have also been observed in a different setting: in \cite{Humphries2025}, Humphries introduced the notion of newform $K$-types for $\GL_n$ over Archimedean local fields, which in general differ from Vogan's minimal $K$-types, and derived propagation formulas relating Whittaker newforms that are likewise normalized via Jacquet integrals.


Another classical approach to obtaining explicit formulas or propagation formulas for Whittaker functions over Archimedean local fields is based on the analysis of the systems of partial differential equations satisfied by these functions. For general linear groups, this approach has led to the results \cite{MIO2004,Miyazaki2009,HO2009,HIM2012,IO2014} mentioned above, as well as the recent work of Hirano--Ishii--Miyazaki \cite{HIM2026} for $\GL_4(\R)$. For other quasi-split reductive groups, explicit formulas obtained via this method include the work of Koseki--Oda \cite{KO1995} for $U(2,1)$; the results of Miyazaki--Oda \cite{MO1993}, Moriyama \cite{Moriyama2002}, and Ishii \cite{Ishii2005} for $\mathrm{GSp}_4(\R)$; and the result of Ishii \cite{Ishii2013} for $\mathrm{SO}_{2n+1}(\R)$.


In contrast to the approaches based on Jacquet integrals or differential equations, we propose a third approach to explicit formulas for Whittaker functions via explicit theta correspondence. While the existence and non-vanishing of theta liftings are usually established by abstract arguments, for generic representations the theta lifting in the almost equal rank case admits a concrete realization in terms of propagation formulas relating Whittaker functions. We carry out this program for reductive dual pairs of type~II, and we expect that similar results should also hold for reductive dual pairs of type~I. For example, we recover the explicit formulas for Whittaker functions of discrete series representations of $\mathrm{GSp}_4(\R)$ and $U(2,1)$ mentioned above by considering the dual pairs $(\mathrm{GO}_{2,2}(\R), \mathrm{GSp}_4(\R))$ and $(U(1,1), U(2,1))$, as studied in \cite{CI2019} and \cite{Chen2025}, respectively.

\subsection{Main results}

Let $F$ be a local field of characteristic zero. Consider the reductive dual pair $(\GL_n(F),\GL_{n+1}(F))$. Let $\omega$ be the Weil representation of $\GL_n(F) \times \GL_{n+1}(F)$ on the Schrodinger model $\sS=\sS(M_{n,n+1}(F))$. Let $\pi$ and $\Pi$ be irreducible generic representations of $\GL_n(F)$ and $\GL_{n+1}(F)$ respectively. We denote by $\sW(\pi,\psi_n)$ and $\sW(\Pi,\psi_{n+1})$ the spaces of associated Whittaker functions. Following is the first main result of this paper.

\begin{thm}[Theorem \ref{T:image}]
Assume $\Pi = \pi \times 1$. Then we have a non-zero intertwining map
\[
\sS \otimes \sW(\pi,\psi_{n}) \longrightarrow \sW(\Pi,\psi_{n+1}),\quad \varphi \otimes W \longmapsto \mathcal{W}_{\varphi\otimes W}
\]
given by the absolutely convergent integral
\[
\mathcal{W}_{\varphi\otimes W}(g) =
\int_{N_n(F)\backslash\GL_n(F)}\int_{N_{n+1}(F)}
\omega(h,ug)\varphi(\varepsilon_n)W(h)\psi^{-1}_{n+1}(u)\,du\,dh.
\]
Here $\varepsilon_n=(I_n,0_{n,1})$.
\end{thm}

Assume that $F$ is Archimedean. Let $K_n$ and $K_{n+1}$ denote the standard maximal compact subgroups of $\GL_n(F)$ and $\GL_{n+1}(F)$, respectively. Let $\tau$ (resp.\,$\Upsilon$) be the minimal $K_n$-type (resp.\,$K_{n+1}$-type) of $\pi$ (resp.\,$\Pi$). 
Let $\sH \subset \sS$ denote the subspace of joint harmonics in the sense of Howe. This is a $(K_n \times K_{n+1})$-invariant subspace that admits a multiplicity-free decomposition. In particular, the $(K_n \times K_{n+1})$-type $\tau^\vee \boxtimes \Upsilon$ occurs in $\sH$.
A minimal $K_n$-type Whittaker function of $\pi$, denoted by $W_\pi$, is a nonzero $\tau^\vee$-valued Whittaker function lying in the one-dimensional space
\[
\left(\sW(\pi,\psi_n)\otimes \tau^\vee\right)^{K_n}.
\]
Similar notation and conventions apply to $\Pi$.
The second main result of this paper is a propagation formula relating $W_\Pi$ to $W_\pi$.

\begin{thm}[Theorem \ref{T:main}]
Assume $\Pi = \pi \times 1$ and $F$ is Archimedean. 
Let 
\[
\varphi_{\tau,\Upsilon} \in \left(\sS \otimes (\tau\boxtimes\Upsilon^\vee)\right)^{K_n \times K_{n+1}} \cap \sH
\]
be a non-zero joint harmonic. Then a minimal $K_{n+1}$-type Whittaker function of $\Pi$ is given by
\begin{align}\label{E:propagation}
\begin{split}
W_{\Pi}(g)
 = |\det(g)|_F^{n/2}
&\int_{T^+_n}\int_{N_{n+1}(F)}\,du\,dt\\
&\quad\langle \varphi_{\tau,\Upsilon}(t^{-1}\varepsilon_n ug), W_{\pi}(t)\rangle_\tau
\psi^{-1}_{n+1}(u)
\delta^{-1}_{B_n(F)}(t)
|\det(t)|_F^{-(n+1)/2}.
\end{split}
\end{align}
Here $\<\cdot,\cdot\>_\tau$ is the natural $K_n$-invariant bilinear pairing between $\tau$ and $\tau^\vee$.
\end{thm}

\begin{rmk}
In practice, one can begin with any $\varphi \in \left(\sS \otimes (\tau\boxtimes\Upsilon^\vee)\right)^{K_n \times K_{n+1}}$ which is not necessary a joint harmonic. If the resulting Whittaker function in (\ref{E:propagation}) is non-zero, then it is a minimal $K_{n+1}$-type Whittaker function of $\Pi$.
\end{rmk}

This paper is organized as follows. In \S\,\ref{S:notation}, we introduce the notation used throughout the article. In \S\,\ref{S:theta}, we review the theta correspondence for reductive dual pairs of type~II and establish our main result on explicit theta lifting (Theorem~\ref{T:image}). In \S\,\ref{S:propagation}, we prove the propagation formulas for minimal $K$-type Whittaker functions over Archimedean local fields (Theorem~\ref{T:main}). In \S\,\ref{SS:Shintani} and \S\,\ref{SS:Miyazaki}, we apply our main results to recover some classical formulas from the literature. In \S\,\ref{SS:new cases}, we derive new explicit formulas for Archimedean Whittaker functions using the propagation formula. Finally, in \S\,\ref{S:Asai}, we compute the Asai local zeta integrals using these new explicit formulas.

\section{Notation}\label{S:notation}

\subsection{}
Throughout this article, let $F$ denote a local field of characteristic zero. Thus, $F$ is either $\C$, $\R$, or a finite extension of 
$\mathbb{Q}_p$. When $F$ is non-Archimedean, denote by $\frak{o}$ the valuation ring of $F$, by $\frak{p}$ the maximal ideal of 
$\frak{o}$, by $\varpi$ a fixed generator of $\frak{p}$, and by $q$ the cardinality of $\frak{o}/\frak{p}$. 

Let $|\cdot|_F$ denote the standard absolute value on $F$. Thus, $|x|_\C=x\bar{x}$, $|x|_\R=\max\{x,-x\}$, and when $F$ is 
non-Archimedean, $|\cdot|_F$ satisfies $|\varpi|_F=q^{-1}$. Fix a non-trivial additive character $\psi$ of $F$. For convenience, 
we assume that $\psi(x)=e^{2\pi \sqrt{-1}(x+\bar{x})}$ when $F=\C$, $\psi(x)=e^{2\pi\sqrt{-1}x}$ when $F=\R$, and that 
$\psi_F$ is trivial on $\frak{o}$ but non-trivial on $\frak{p}^{-1}$ when $F$ is non-Archimedean. 

Let $M_{n,m}(F)$ denote the space of $n$-by-$m$ matrices with entries in $F$. The identity matrix in 
$M_{n,n}(F)$ is denoted by $I_n$, while the zero matrix in $M_{n,m}(F)$ is denoted by $0_{n,m}$. We write $\sS(M_{n,m}(F))$ for the 
space of Bruhat-Schwartz functions on $M_{n,m}(F)$.

\subsection{}
Let $B_n=T_nN_n$ denote the upper triangular Borel subgroup of $\GL_n$, where $T_n$ is the diagonal torus and $N_n$ is its unipotent 
radical. Define a generic character $\psi_n:N_n(F)\longrightarrow\mathbb{C}^{\times}$ by 
\begin{align}\label{E:psi_n}
u=(u_{ij})\mapsto \psi(u_{12}+u_{23}+\cdots+ u_{n-1,n}).
\end{align}

Denote by $F^+$ the quotient of $F^\times$ by its subgroup consisting of elements whose absolute value equals $1$.  We identify 
$F^+$ with $\R_+$ when $F$ is Archimedean, and with $\langle\varpi\rangle\simeq\mathbb{Z}$ when $F$ is non-Archimedean. 
Let $T_n^+$ denote the subgroup of $T_n(F)$ consisting of the diagonal matrices whose diagonal entries lie in $F^+$. 

Let $K_n\subset\GL_n(F)$ denote the standard maximal compact subgroup, that is, $K_n={\rm U}_n(\mathbb{R})$ when 
$F=\C$, $K_n={\rm O}_n(\mathbb{R})$ when $F=\R$, and $K_n=\GL_n(\frak{o})$ when $F$ is non-Archimedean. 
Note that we have the Iwasawa decomposition $\GL_n(F)=N_n(F)\,T^+_n\,K_n$.

\subsection{}
In this article, a representation of $\GL_n(F)$ refers to a smooth admissible Fr\'echet representation of moderate 
growth in the sense of Casselman-Wallach (\cite{Casselman1989}) when $F$ is Archimedean, and to a smooth representation of 
finite length when $F$ is non-Archimedean. We often abuse notation by using the same symbol to denote both a representation and 
the space on which it acts. If $\pi$ is a representation of $\GL_n(F)$, its contragredient is denoted by $\pi^\vee$. When $F$ is 
Archimedean, by a $K$-type of $\pi$, we mean a $K_n$-type of $\pi$. 

If $\pi_1,...,\pi_r$ are representations of $\GL_{n_1}(F),..., \GL_{n_r}(F)$, respectively, with $n_1+\cdots+n_r=n$, 
we denote by
\[
\pi_1\times\cdots\times\pi_r
\]
the representation of $\GL_n(F)$ parabolically induced
\footnote{All inductions in this article are normalized.}
 from $\pi_1\widehat{\boxtimes}\cdots\widehat{\boxtimes}\pi_r$ with respect to 
the standard parabolic subgroup of $\GL_n(F)$ of block upper triangular matrices, whose Levi subgroup is isomorphic to 
\[
\GL_{n_1}(F)\,\times\cdots\times\,\GL_{n_r}(F).
\] 
Here and after, $\widehat{\boxtimes}$ denotes the completed projective tensor product when $F$ is Archimedean, and the 
usual (external) tensor product when $F$ is non-Archimedean. 

\subsection{}
Let $\pi$ be a representation of $\GL_n(F)$. A Whittaker functional of $\pi$ (with respect to $\psi_n$) is a continuous 
functional $\lambda_{\psi_n}:\pi\rightarrow\mathbb{C}$ that satisfies 
\[
\lambda_{\psi_n}(\pi(u)v)=\psi_n(u)\lambda_{\psi_n}(v)
\quad
\text{for $u\in N_n(F)$ and $v\in \pi$.}
\]
If the space of Whittaker functionals on $\pi$ is one-dimensional, we say that $\pi$ is 
generic, and denote by $\mathscr{W}(\pi,\psi_n)$ the Whittaker model of $\pi$; it consists of the functions 
\[
W_v(h)=\lambda_{\psi_n}(\pi(h)v)\quad
\text{for $h\in\GL_n(F)$ and $v\in\pi$}.
\]
The action of $\GL_n(F)$ on $\sW(\pi,\psi_n)$ is given by the right translation $\rho$.
It is well-known that when $\pi$ is irreducible, the space of Whittaker functionals on $\pi$ is at most one-dimensional.

Let $\mathcal{C}^\infty(N_n(F)\backslash\GL_n(F);\psi_n)^{{\rm mg}}$ denote the space of smooth functions 
\[
W:\GL_n(F)\rightarrow\mathbb{C}
\]
satisfying:
\begin{itemize}
\item $W(uh)=\psi_n(u)W(h)$ for $u\in N_n(F)$ and $h\in\GL_n(F)$;
\item when $F$ is Archimedean, $W$ is of moderate growth, namely, for every $X$ in the universal enveloping algebra of 
${{\rm Lie}}(\GL_n(F))$, there exist $C>0$ and $d\ge 0$ such that 
\[
|XW(h)|_\C\le C\|h\|^d,
\]
for every $h\in\GL_n(F)$, where
\[
\|h\|:={\rm tr}({}^t\bar{h} h)^{1/2}+{\rm tr}({}^t\bar{h}^{-1}h^{-1})^{1/2}.
\]
\end{itemize}
Note that 
\[
\sW(\pi,\psi_n)\subset\mathcal{C}^\infty(N_n(F)\backslash\GL_n(F);\psi_n)^{{\rm mg}}.
\]

Let $\sC(N_n(F)\backslash\GL_n(F);\psi_n^{-1})$ denote the space of smooth functions 
\[
f: \GL_n(F) \rightarrow \C
\]
satisfying:
\begin{itemize}
\item $f(uh)=\psi_n^{-1}(u)f(h)$ for $u\in N_n(F)$ and $h\in\GL_n(F)$;
\item when $F$ is Archimedean, $f$ is rapidly decreasing modulo $N_n(F)$, namely, for every $X$ in the universal enveloping algebra of 
${{\rm Lie}}(\GL_n(F))$ and $d \geq 0$, we have
\[
\sup_{k \in K_n,t\in T_n}|Xf(tk)|_\C\|t\|^d < \infty;
\]
\item when $F$ is non-Archimedean, $f$ has compact support modulo $N_n(F)$.
\end{itemize}

\section{Theta correspondence}\label{S:theta}
\subsection{Type II dual pairs}
Let $n,m$ be two positive integers.
The pair $(\GL_n, \GL_m)$ forms a type II irreducible reductive dual pair in the sense of Howe 
(\cite{Howe1979}, \cite{MVW1987}). Consequently, this pair admits a Weil representation $\left(\omega_{n,m}, \mathscr{S}(M_{n,m}(F))\right)$, whose 
action is given by 
\[
\omega_{n,m}(h,g)\varphi(x)
=
|\det(h)|_F^{-m/2}|\det(g)|_F^{n/2}\varphi(h^{-1}xg),
\]
where $(h,g)\in\GL_n(F)\times\GL_m(F)$ and $\varphi\in\sS(M_{n,m}(F))$.

Now let $\pi$ be an irreducible representation of $\GL_n(F)$, and assume that $\pi$ appears as a quotient of $\omega_{n,m}$. 
A famous result of Howe (\cite{Howe1989}, \cite{MVW1987}) asserts that there is a unique irreducible representation 
$\sigma:=\theta_{n,m}(\pi)$ of $\GL_m(F)$ such that 
\begin{equation}\label{E:theta}
{\rm Hom}_{\GL_n(F)\times\GL_m(F)}(\omega_{n,m},\pi\widehat{\boxtimes}\sigma)\neq 0.
\end{equation}
Moreover, the Hom-space in \eqref{E:theta} is one-dimensional, and $\pi$ and $\sigma$ uniquely determine each other. Therefore, 
the condition \eqref{E:theta} establishes a bijection between the sets of irreducible representations of $\GL_n(F)$ and $\GL_m(F)$
occurring as quotients of $\omega_{n,m}$. 

\subsection{Abstract theta lifts}
A more delicate problem is to describe this bijection explicitly, namely, to determine which $\pi$ occurs as a quotient of 
$\omega_{n,m}$, and for such a $\pi$, to describe $\theta_{n,m}(\pi)$ explicitly. The term ``explicitly" can have several 
interpretations. Ideally, the bijection could be computed in terms of some parameterization of the representations, such as the 
Langlands classification or its variations. However, other descriptions are possible. For the dual pair $(\GL_n, \GL_m)$, this 
problem was addressed by M\oe glin (\cite[Proposition III 9]{Moeglin1989}) and Adams-Barbasch 
(\cite[Theorem 2.6]{AdamsBarbasch1995}) when $F$ is Archimedean, and by Minguez (\cite[Theorem 1]{Minguez2008}) 
when $F$ is non-Archimedean. 

In this article, we are particularly interested in the case $m = n+1$, which we assume from this point onward. In this case, 
we also write 
\begin{equation}\label{E:case m=n+1}
\omega=\omega_{n,n+1},\quad\theta=\theta_{n,n+1}\quad\text{and}\quad\sS=\sS(M_{n,n+1}(F)).
\end{equation}
Now, we have the following result:

\begin{thm}[M\oe glin, Admas-Barbasch and Minguez]\label{T:theta}
Every irreducible representation $\pi$ of $\GL_n(F)$ occurs as a quotient of $\omega$. Moreover, when $F$ is Archimedean, 
$\theta(\pi)$ is the unique subquotient of $\pi^\vee\times 1$ that containing the minimal $K$-type of $\pi^\vee\times 1$.
On the other hand, when $F$ is non-Archimedean and $\pi$ is the Langlands quotient of $\pi_1\times\cdots\times\pi_r$, where
each $\pi_j$ is an essentially square-integrable representation, then $\theta(\pi)$ is the Langlands quotient
\footnote{We follow the convention stated at the begging of \cite[Section 6]{Minguez2008}.}
of 
\[
1\times\pi_r^\vee\times\cdots\times\pi_1^\vee.
\]
Here, $1$ denotes the trivial representation of $\GL_1(F)$.
\end{thm}

As a corollary, we obtain:

\begin{cor}\label{C:theta}
Let $\pi$ be an irreducible generic representation of $\GL_n(F)$. Suppose that $\pi^\vee\times 1$ is irreducible. Then $\theta(\pi)=\pi^\vee\times 1$. 
\end{cor}

\begin{proof}
This follows immediately from the irreducibility assumption and the standard module conjecture for $\GL_n$. 
\end{proof}

\subsection{Whittaker functionals on the big theta lifts à la Gomez–Zhu}

Let $\pi$ be an irreducible representation of $\GL_n(F)$.  
Since $\pi^\vee$ occurs as a quotient of $\omega$ by Theorem \ref{T:theta}, the big theta lift
\footnote{When $F$ is Archimedean, the big theta lift is defined as the maximal Hausdroff quotient of 
$\left(\omega\,\widehat{\otimes}\,\pi\right)_{\GL_n(F)}$. However, a recent preprint of Geng-Xue (\cite{GengXue2026}) shows that 
$\left(\omega\,\widehat{\otimes}\,\pi\right)_{\GL_n(F)}$ is already Hausdroff.}
$\Theta(\pi^\vee)$ of $\pi^\vee$ defined by 
\[
\Theta(\pi^\vee):=\left(\omega\,\widehat{\otimes}\,\pi\right)_{\GL_n(F)}
\]
is a non-zero smooth representation of $\GL_{n+1}(F)$.
We have the following result on the space of Whittaker functionals on $\Theta(\pi^\vee)$.

\begin{prop}\label{P:GZ}
We have
\[
\dim_\C{\rm Hom}_{N_{n+1}(F)}\left(\Theta(\pi^\vee),\psi_{n+1}\right) = \dim_\C{\rm Hom}_{N_{n}(F)}\left(\pi,\psi_{n}\right).
\]
\end{prop}

In particular, $\pi$ is generic if and only if $\Theta(\pi^\vee)$ is generic. 
First, note that by \cite[Proposition 4.9]{GomezZhu2014}, whose proof carries over mutatis mutandis to our setting, we have an $\GL_n(F)$-intertwining isomorphism
\[
\sC(N_n(F)\backslash \GL_n(F);\psi_n^{-1})_{\GL_n(F),\,\pi^\vee} \simeq {\rm Hom}_{N_n(F)}(\pi,\psi_n)^* \otimes \pi^\vee,
\]
where the left-hand side is the maximal $\pi^\vee$-isotypic quotient and ${\rm Hom}_{N_n(F)}(\pi,\psi_n)^*$ is the algebraic dual of ${\rm Hom}_{N_n(F)}(\pi,\psi_n)$. 
Assume that there exists an $\GL_n(F)$-equivariant isomorphism 
\begin{align}\label{E:GZ0}
\sS(M_{n,n+1}(F))_{N_{n+1}(F),\,\psi_{n+1}} \simeq \sC(N_n(F)\backslash \GL_n(F);\psi_n^{-1}).
\end{align}
Then we have
\[
\left(\sS(M_{n,n+1}(F))_{N_{n+1}(F),\,\psi_{n+1}}\right)_{\GL_n(F),\,\pi^\vee} \simeq {\rm Hom}_{N_n(F)}(\pi,\psi_n)^* \otimes \pi^\vee.
\]
On the other hand,
\begin{align*}
\left(\sS(M_{n,n+1}(F))_{N_{n+1}(F),\,\psi_{n+1}}\right)_{\GL_n(F),\,\pi^\vee} &\simeq \left(\sS(M_{n,n+1}(F))_{\GL_n(F),\,\pi^\vee}\right)_{N_{n+1}(F),\,\psi_{n+1}}\\
& \simeq \Theta(\pi^\vee)_{N_{n+1}(F),\,\psi_{n+1}} \otimes \pi^\vee.
\end{align*}
We thus conclude that
\begin{align*}
\dim_\C{\rm Hom}_{N_{n+1}(F)}\left(\Theta(\pi^\vee),\psi_{n+1}\right) & = \dim_\C \left(\Theta(\pi^\vee)_{N_{n+1}(F),\,\psi_{n+1}}\right)^*\\
& = \dim_\C{\rm Hom}_{N_{n}(F)}\left(\pi,\psi_{n}\right).
\end{align*}
Therefore, the proof of Proposition~\ref{P:GZ} reduces to the construction of an isomorphism~\eqref{E:GZ0}.
An explicit defined isomorphism is given as follows: 
\begin{align}\label{E:explicit iso.}
\sS(M_{n,n+1}(F))_{N_{n+1}(F),\,\psi_{n+1}} \rightarrow  \sC(N_n(F)\backslash \GL_n(F);\psi_n^{-1}),\quad [\varphi] \mapsto {\Psi}(\varphi)
\end{align}
where
\begin{align}\label{E:Psi}
{\Psi}(\varphi;h):= \int_{N_{n+1}(F)}\omega_{n,n+1}(h,u)\varphi(\varepsilon_n)\psi_{n+1}^{-1}(u)du \quad\text{for $h \in \GL_n(F)$}.
\end{align}
Here $\varepsilon_n:=(I_n,0_{n,1})$.
In the case of a type I reductive dual pair, an explicit isomorphism in the analogous setting was given in \cite[(6.20) and Proposition~6.5]{GomezZhu2014}. As noted in \cite[Remark~2.2]{GomezZhu2014}, the authors expect that the same statement should remain valid for type II reductive dual pairs.
In the following, we carry out this construction in the type II setting and show that it coincides with the one defined in (\ref{E:explicit iso.}). Following the notation and framework of \cite{GomezZhu2014}, we identify the corresponding objects for type II reductive dual pairs and non-degenerate Whittaker functionals, and define the required isomorphism. 
Once these identifications are made, the verification of the remaining details then proceeds exactly as in \cite[\S\,6.4 and \S\,6.5]{GomezZhu2014}, and will not be repeated.


Let $D = F\oplus F$ be the split quadratic algebra over $F$.
Fix two signatures $\varepsilon, \tilde{\varepsilon} \in \{1,-1\}$.
Let $(V,B)$ and $(\tilde{V},\tilde{B})$ be the $\varepsilon$-Hermitian and $\tilde{\varepsilon}$-Hermitian spaces over $D$ defined by
\[
V := D^{n \times 1};\quad \tilde{V} := D^{(n+1) \times 1}
\]
and
\[
B(w,v) := {}^twQv\quad \text{for $w,v \in V$};\quad \tilde{B}(\tilde{w},\tilde{v}) := {}^t\tilde{w}\tilde{Q}\tilde{v}\quad \text{for $\tilde{w},\tilde{v} \in \tilde{V}$},
\]
where
\[
Q:=I_n \oplus \varepsilon I_n \in \GL_n(D);\quad \tilde{Q}:=I_{n+1} \oplus \tilde{\varepsilon} I_{n+1} \in \GL_{n+1}(D).
\]
Let $G=G(V,B) \subset \GL_D(V)$ and $\tilde{G}=G(\tilde{V},\tilde{B}) \subset \GL_D(\tilde{V})$ be the isometry groups of the Hermitian spaces.
Upon projecting to the first components, we identify $G$ and $\tilde{G}$ with $\GL_n(F)$ and $\GL_{n+1}(F)$ respectively. More precisely, we have isomorphisms 
\[
G \rightarrow \GL_n(F),\quad (g,{}^tg^{-1})\mapsto g;\quad \tilde{G} \rightarrow \GL_{n+1}(F),\quad (\tilde{g},{}^t\tilde{g}^{-1})\mapsto \tilde{g}.
\]
Similarly, we identify the Lie algebras $\frak{g}$ and $\tilde{\frak{g}}$ of $G$ and $\tilde{G}$ with $M_{n,n}(F)$ and $M_{n+1,n+1}(F)$ respectively via the isomorphisms
\[
\frak{g} \rightarrow M_{n,n}(F),\quad (X,-{}^tX)\mapsto X;\quad \tilde{\frak{g}} \rightarrow M_{n+1,n+1}(F),\quad (\tilde{X},-{}^t\tilde{X})\mapsto \tilde{X}.
\]
Under the above identification, the bilinear pairing $\kappa$ on $\frak{g}$ in \cite[(2.1)]{GomezZhu2014} is equal to
\begin{align}\label{E:big theta 1}
\kappa(T,S) = -{\rm tr}(TS)\quad\text{for $T,S \in M_{n,n}(F)\simeq \frak{g}$}.
\end{align}
Similarly for the bilinear pairing pairing $\tilde{\kappa}$ on $\tilde{\frak{g}}$.
We assume that $\varepsilon \tilde{\varepsilon} = 1$. Then $(G,\tilde{G})$ is a reductive dual pair of type II.
Let $\gamma = \{H,X,Y\}\subset \frak{g}$ be an $\frak{sl}_2$-triple defined by 
\[
X:=-\sum_{i=1}^{n-1}E_{i+1,i},\quad Y:=-\sum_{i=1}^{n-1}i(n-i)E_{i,i+1},\quad H:=-\sum_{i=1}^n (n-2i+1)E_{i,i}.
\]
Here $E_{i,j}$ is the matrix unit with non-zero $(i,j)$-entry.
In this case, the weight space decomposition of $\frak{g}$ in \cite[(3.1)]{GomezZhu2014} under the action of $H$ is given by
\[
\frak{g} = \bigoplus_{i=1-n}^{n-1}\frak{g}_{2i}.
\]
In particular, it is easy to see that the associated unipotent subgroups $N,U$ and the parabolic subgroup $P$ of $G$ defined as in \cite[\S\,3.1]{GomezZhu2014} are equal to
\[
N=U=N_n(F),\quad P=B_n(F).
\]
Also, the weight space decomposition of $V$ in \cite[(3.2)]{GomezZhu2014} under the action of $H$ is given by
\[
V = \bigoplus_{j=0}^{n-1} V_{n-1-2j}.
\]
We set $V_i = 0$ if $i \in \Z$ does not appear in the above index.
Note that 
\[
V_{n-1-2j} = D(e_{n-j} \oplus e_{j+1})\quad \text{for $0 \leq j \leq n-1$},
\]
where $\{e_1,\cdots,e_n\}$ is the standard basis of $F^{n \times 1}$. 
It follows from (\ref{E:big theta 1}) that the associated additive character $\chi_\gamma$ of $U=N_n(F)$ in \cite[(1.1) and \S\,3.3]{GomezZhu2014} is equal to the character $\psi_n$ defined in (\ref{E:psi_n}).
Similarly, let $\tilde{\gamma} = \{\tilde{H},\tilde{X},\tilde{Y}\} \subset \tilde{\frak{g}}$ be the $\frak{sl}_2$-triple defined as above with $n$ replaced by $n+1$.
Then we have
\[
\tilde{\frak{g}} = \bigoplus_{i=-n}^{n}\tilde{\frak{g}}_{2i},\quad \tilde{V} = \bigoplus_{j=0}^{n} \tilde{V}_{n-2j}
\]
and
\[
\tilde{N}=\tilde{U}=N_{n+1}(F),\quad \tilde{P}=B_{n+1}(F),\quad \chi_{\tilde{\gamma}} = \psi_{n+1}.
\]
In the notation of \cite[(6.14)]{GomezZhu2014}, we have $r=n-1$.
As in \cite[Definition 5.5]{GomezZhu2014}, we say that a $D$-linear homomorphism $T \in {\rm Hom}_D(V,\tilde{V})$ lifts $\gamma$ to $\tilde{\gamma}$ if $T^*T = X$, $TT^* = \tilde{X}$, and $T(V_j) \subset \tilde{V}_{j+1}$ for all $j$. Here $T^* \in {\rm Hom}_D(\tilde{V},V)$ is the adjoint of $T$, that is, 
\[
B(v,T^*\tilde{v}) = \tilde{B}(Tv,\tilde{v})\quad \text{for $v\in V$, $\tilde{v} \in \tilde{V}$}.
\]
If we require further that $T$ is injective, then $T$ is essentially unique as proved in \cite[Lemma 5.7]{GomezZhu2014}.
We fix an explicit choice as follows: 
With respect to the standard bases, we identify ${\rm Hom}_F(F^{n \times 1},F^{(n+1) \times 1})$ with $M_{n+1,n}(F)$. Then
\begin{align}\label{E:identification}
{\rm Hom}_D(V,\tilde{V}) = M_{n+1,n}(F) \oplus M_{n+1,n}(F).
\end{align}
We define
\begin{align}\label{E:moment}
T:= \bp 0_{1,n} \\ I_n \ep \oplus \bp I_n \\ 0_{1,n} \ep \in  {\rm Hom}_D(V,\tilde{V}).
\end{align}
Following the notation in \cite[Lemma 5.6]{GomezZhu2014}, we denote $T_j:=T\vert_{V_j} \in {\rm Hom}_D(V_j,\tilde{V}_{j+1})$.

Now we begin the construction of an isomorphism (\ref{E:GZ0}). 
Let $(\mathbb{W},\<\cdot,\cdot\>)$ be the symplectic space over $F$ of dimension $2n(n+1)$ defined by
\[
\mathbb{W} := {\rm End}_D(V,\tilde{V}) = M_{n+1,n}(F) \oplus M_{n+1,n}(F)
\]
and
\[
\<T,S\>:={\rm tr}\left({}^tT_1S_2-{}^tT_2S_1\right) \quad \text{for $T=T_1 \oplus T_1,\,S = S_1 \oplus S_2$}.
\]
Let $\mathbb{X} := 0\oplus M_{n+1,n}(F)$ and $\mathbb{Y} := M_{n+1,n}(F)\oplus 0$. Then $\mathbb{W} = \mathbb{X} + \mathbb{Y}$ gives a complete polarization of $\mathbb{W}$.
The natural homomorphism from $G \times \tilde{G}$ to ${\rm Sp}(\mathbb{W})$ factors through the Siegel parabolic subgroup that stabilizing $\mathbb{Y}$. Hence the Weil representation $\omega_{V,\tilde{V}}$ of $G \times \tilde{G}$ on $\sS(\mathbb{X})$ (the Schrodinger model) is given by
\[
\omega_{V,\tilde{V}}(h,g)\varphi(x)
=
|\det(h)|_F^{-(n+1)/2}|\det(g)|_F^{n/2}\varphi({}^tgx{}^th^{-1})
\]
for $(h,g) \in G \times \tilde{G}$ and $\varphi \in \sS(\mathbb{X})$.
Note that the transpose defines an intertwining isomorphism
\[
(\omega_{n,n+1},\sS(M_{n,n+1}(F))) \rightarrow (\omega_{V,\tilde{V}},\sS(\mathbb{X})),\quad \varphi \mapsto {}^t\varphi
\]
where ${}^t\varphi(x):=\varphi({}^tx)$.
In \cite[p.\,831]{GomezZhu2014}, the Weil representation with the Schrodinger model is denoted by $(\omega_{(r),(r+1)},\sY_{(r),(r+1)})$.
Following the notation in \cite[(6.14)]{GomezZhu2014}, for $0 \leq l \leq r$ and $0 \leq m \leq r+1$, let
\[
V_{(l)} := \bigoplus_{k=-l}^l V_k,\quad \tilde{V}_{(m)} := \bigoplus_{k=-m}^m \tilde{V}_k.
\]
Consider another complete polarization $\mathbb{W} = \mathbb{X}' + \mathbb{Y}'$, where
\begin{align*}
\mathbb{X}' &:= {\rm Hom}_D(V_{(r)},\tilde{V}_{r+1}) \oplus {\rm Hom}_D(V_{-r},\tilde{V}_{(r)}) \oplus ({\rm Hom}_D(V_{(r-1)},\tilde{V}_{(r)})\cap \mathbb{X}),\\
\mathbb{Y}' &:= {\rm Hom}_D(V_{(r)},\tilde{V}_{-r-1}) \oplus {\rm Hom}_D(V_{r},\tilde{V}_{(r)}) \oplus ({\rm Hom}_D(V_{(r-1)},\tilde{V}_{(r)})\cap \mathbb{Y}).
\end{align*}
With respect to this change of polarizations, we have the intertwining isomorphism
\begin{align}\label{E:change of polarization}
(\omega_{V,\tilde{V}},\sS(\mathbb{X})) \rightarrow (\hat{\omega}_{V,\tilde{V}},\sS(\mathbb{X}')),\quad \varphi\mapsto \hat{\varphi}
\end{align}
defined by the partial Fourier transform
\[
\hat{\varphi}(x'):= \int_{\mathbb{Y}' / \mathbb{Y}\cap\mathbb{Y}'}\varphi(x) \psi\left(\tfrac{1}{2}\<x,y\> - \tfrac{1}{2}\<x',y'\>\right)dy' \quad \text{for $x' \in \mathbb{X}'$},
\]
where $x \in \mathbb{X}$ and $y \in \mathbb{Y}$ are uniquely determined by $x+y = x'+y'$, and the Haar measure is normalized so that $\hat{\omega}_{V,\tilde{V}}$ is unitary.
For a type I reductive dual pair, the analogous isomorphism was mentioned implicitly in \cite[(6.17)]{GomezZhu2014}.
Note that the subgroups ${U}_{r} \subset {U} = N_{n}(F)$ and $G_{(r-1)} \subset G$ defined as in \cite[p.\,834 and (6.23)]{GomezZhu2014} are given by
\[
U_{r} = \left\{\bp 1 & * & * \\ 0 & I_{n-2} & * \\ 0&0&1 \ep\right\},\quad  G_{(r-1)} = \left\{ \bp 1 &0 & 0 \\ 0 & * & 0 \\ 0 & 0 & 1 \ep \right\} \simeq \GL_{n-2}(F).
\]
Similarly for $\tilde{U}_{r+1} \subset \tilde{U} = N_{n+1}(F)$ which is defined as above by replacing $n$ by $n+1$.
Proceeding exactly as in \cite[Lemma 6.6]{GomezZhu2014}, we have an $G$-intertwining isomorphism
\begin{align}\label{E:GZ lemma}
\sS(\mathbb{X})_{\tilde{U}_{r+1},\psi_{n+1}} \rightarrow \sC\left(U_rG_{(r-1)}\backslash G;  \sS({\rm Hom}_D(V_{(r-1)},\tilde{V}_{(r)})\cap \mathbb{X})\right),\quad [\varphi]\mapsto \Psi_r(\varphi)
\end{align}
where 
\[
\Psi_r(\varphi;h)(x):= \hat{\omega}_{V,\tilde{V}}(h,1)\hat{\varphi}(T_r \oplus T_{-r} \oplus x)
\]
for $h \in G$ and $x \in {\rm Hom}_D(V_{(r-1)},\tilde{V}_{(r)})\cap \mathbb{X}$.
Note that $T_r \in {\rm Hom}_D(V_{(r)},\tilde{V}_{r+1})$ and $T_{-r} \in {\rm Hom}_D(V_{-r},\tilde{V}_{(r)})$.
We reinterpret the isomorphism in terms of our model $(\omega_{n,n+1},\sS(M_{n,n+1}(F)))$ as follows. We have isomorphisms
\begin{align*}
F^{2 \times n}  &\rightarrow {\rm Hom}_D(V_{(r)},\tilde{V}_{r+1}),\quad \bp x_1 \\ x_1' \ep \mapsto  \bp 0_{n,n} \\ x_1 \ep \oplus \bp x_1' \\ 0_{n,n} \ep\\
F^{(n-1) \times 2} &\rightarrow {\rm Hom}_D(V_{-r},\tilde{V}_{(r)}),\quad \bp x_2 & x_2'\ep \mapsto  \left(
\begin{array}{c}
0_{1,n} \\ \hline
x_2 \quad 0_{n-1,n-1} \\ \hline
0_{1,n}
\end{array}
\right) \oplus \left(
\begin{array}{c}
0_{1,n} \\ \hline
0_{n-1,n-1} \quad x_2' \\ \hline
0_{1,n}
\end{array}
\right)    \\
M_{n-1,n-2}(F) &\rightarrow {\rm Hom}_D(V_{(r-1)},\tilde{V}_{(r)})\cap \mathbb{X},\quad x \mapsto 0 \oplus \left(
\begin{array}{ccc}
 & 0_{1,n} & \\ \hline
0_{n-2,1} &  x  & 0_{n-2,1}\\ \hline
 &   0_{1,n}  &
\end{array}
\right).
\end{align*}
Note that the right-hand sides are subspaces of $M_{n+1,n}(F) \oplus M_{n+1,n}(F)$ under the isomorphism (\ref{E:identification}).
With respect to these isomorphisms, by taking transpose of matrices, (\ref{E:change of polarization}) can be rewrite as
\[
(\omega_{n,n+1},\sS(M_{n,n+1}(F))) \rightarrow (\hat{\omega}_{n,n+1},\sS(F^{n \times 2} \oplus F^{2 \times (n-1)} \oplus M_{n-2,n-1}(F))),\quad \varphi\mapsto \hat{\varphi}
\]
where
\[
\hat{\varphi}\left( \bp x_1 & x_1' \ep  \oplus \bp x_2 \\ x_2' \ep \oplus x \right) = \int_{F^{n\times 1}}dy_1\int_{F^{1\times (n-1)}}dy_2\,
\varphi\left(
\begin{array}{c|c|c}
  & y_2 &  \\
 x_1' & x &  y_1 \\
  & x_2' & 
\end{array}
\right)\psi^{-1}({}^tx_1y_1+x_2{}^ty_2).
\]
By the explicit definition of $T_r$ and $T_{-r}$ from (\ref{E:moment}), we have
\[
F^{2 \times n}\ni \bp 0_{1,n-1}&  1 \\ 1 &  0_{1,n-1} \ep \mapsto T_r,\quad F^{(n-1) \times 2} \ni \bp 1 & 0_{n-2,1} \\ 0_{n-2,1} & 1 \ep \mapsto T_{-r}. 
\]
Thus, we can rewrite (\ref{E:GZ lemma}) as
\[
\sS(M_{n,n+1}(F))_{\tilde{U}_{r+1},\psi_{n+1}} \rightarrow \sC\left(U_rG_{(r-1)}\backslash G;  \sS(M_{n-2,n-1}(F))\right),\quad [\varphi]\mapsto \Psi_r(\varphi)
\]
where
\[
\Psi_r(\varphi;h)(x) = \int_{F^{n\times 1}}dy_1\int_{F^{1\times (n-1)}}dy_2\, \omega_{n,n+1}(h,1)\varphi \left(
\begin{array}{c|c|c}
 1 & y_2 &  \\
 0_{n-2,1} & x &  y_1 \\
 0 & 0_{n-2}\quad 1 & 
\end{array}
\right)\psi_{n+1}^{-1}(y_{1n}+y_{21})
\]
for $h \in \GL_n(F)$ and $x \in M_{n-2,n-1}(F)$.
Note that $U_rG_{(r-1)}$ is the trivial group when $n=1$.
For $n\geq 2$, the action of $U_rG_{(r-1)}$ on $\sS(M_{n-2,n-1}(F))$ is given by
\begin{align*}
\bp 1 & w & z \\ 0 & I_{n-2} & y \\ 0&0&1 \ep\cdot\varphi(x) &= \begin{cases}
\psi^{-1}(z)\varphi(x) & \mbox{if $n=2$},\\
\psi^{-1}\left(wx \bp  1 \\ 0_{n-2,1} \ep\right) \varphi\left(x-y\bp 0_{1,n-2} & 1 \ep \right) & \mbox{if $n > 2$},
\end{cases}\\
{\rm diag}(1,h,1)\cdot\varphi &= \omega_{n-2,n-1}(h,1)\varphi.
\end{align*}
It is then readily seen that the intertwining map $\Psi$ in (\ref{E:Psi}) coincides with $\Psi_0$ when $n=1$, and with $\Psi_1$ when $n=2$. 
Proceeding inductively, we conclude that (recall that $r=n-1$)
\[
\Psi =
\begin{cases}
\Psi_0 \circ \Psi_2 \circ \cdots \circ \Psi_r, & \text{if $n$ is odd},\\[4pt]
\Psi_1 \circ \Psi_3 \circ \cdots \circ \Psi_r, & \text{if $n$ is even}.
\end{cases}
\]
At each stage, the map $\Psi_i$ factors through a suitable coinvariant quotient; consequently, $\Psi$ defines the isomorphism \eqref{E:explicit iso.}. For a type I reductive dual pair, the analogous inductive structures are described in \cite[(6.18) and (6.20)]{GomezZhu2014}.

\subsection{Explicit theta lifts}
One approach to study $\theta(\pi^\vee)$ is to explicitly construct a non-zero element in the Hom-space in \eqref{E:theta}. 
When $\pi$ is generic, this was achieved by Watanabe in \cite{Watanabe1995}, where he computed the Fourier coefficients 
of a global theta lift of a cusp form on $\GL_{n+1}$. In particular, he obtained a map
\[
\mathcal{W}:\mathscr{S}\otimes\mathscr{W}(\pi,\psi_n)\longrightarrow
\mathcal{C}^\infty(N_{n+1}(F)\backslash GL_{n+1}(F);\psi_{n+1})^{{\rm mg}},
\]
which is given by the following absolutely convergent integral:
\begin{align}\label{E:V}
\begin{split}
\mathcal{W}_{\varphi\otimes W}(g) 
& =
\int_{N_n(F)\backslash\GL_n(F)}\int_{N_{n+1}(F)}
\omega(h,ug)\varphi(\varepsilon_n)W(h)\psi^{-1}_{n+1}(u)\,du\,dh\\
&\quad = \int_{N_n(F)\backslash\GL_n(F)} \Psi(\omega(1,g)\varphi;h)W(h)dh.
\end{split}
\end{align}
where $\varphi\in\sS$, $W\in\sW(\pi,\psi_n)$ and $\varepsilon_n:=(I_n,0_{n,1})\in M_{n,n+1}(F)$. 

It is easy to verify that 
\begin{equation}\label{E:Whittaker equivariant}
\mathcal{W}_{\omega(h,g)\varphi\otimes \rho(h)W}
=
\rho(g)\mathcal{W}_{\varphi\otimes W}
\quad\text{for $(h,g)\in\GL_n(F)\times\GL_{n+1}(F)$.}
\end{equation}
It follows that $\GL_{n+1}(F)$ acts on the space
\[
\mathscr{V}(\pi,\psi_{n+1})
:=
\biggl\{
\mathcal{W}_{\varphi\otimes W}\mid \text{$\varphi\in\mathscr{S}$ and $W\in\mathscr{W}(\pi,\psi_n)$}
\biggl\}
\subset
\mathcal{C}^\infty(N_{n+1}(F)\backslash\GL_{n+1}(F);\psi_{n+1})^{{\rm mg}}
\]
by the right translation, and the spaces $\mathscr{V}(\pi,\psi_{n+1})$ and $\mathscr{W}(\theta(\pi^\vee),\psi_{n+1})$ should, in 
some way, be related. We prove that this is in fact the case; to this end, we need the following lemma. 

\begin{lm}\label{L:image}
The space $\sV(\pi,\psi_{n+1})$ is non-zero.
\end{lm}

\begin{proof}
Our task is to show that there exist $W\in\sW(\pi,\psi_n)$ and $\varphi\in\sS$ such that $\mathcal{W}_{\varphi\otimes W}\ne 0$. In fact, 
we can prove a stronger result: given $0\ne W\in\sW(\pi,\psi_n)$, there exists $\varphi\in\sS$ such that $\mathcal{W}_{\varphi\otimes W}\ne 0$. 
Note that, by \eqref{E:Whittaker equivariant}, it suffices to show that $\mathcal{W}_{\varphi\otimes W}\left(I_{n+1}\right)\ne 0$ for some 
$\varphi$. We argue by contradiction. Hence, assume that $\mathcal{W}_{\varphi\otimes W}\left(I_{n+1}\right)=0$ for every $\varphi\in\sS$. 

Let $\sS(\GL_n(F))$ denote the space of $\C$-valued Schwartz functions on $\GL_n(F)$. Thus, it consists of smooth functions 
$\phi:\GL_n(F)\to\C$ satisfying 
\[
\sup_{h\in\GL_n(F)}|X\phi(h)|_\C\|h\|^d<\infty
\]
for every $X$ in the universal enveloping algebra of ${{\rm Lie}}(\GL_n(F))$ and every $d\in\mathbb{N}$ when $F$ is Archimedean
\footnote{The definitions in \cite{Casselman1989B} and \cite{AG2008} coincide; see \cite[\S 1.4]{AG2008}.} 
(\cite{Casselman1989B}, \cite{AG2008}), while it consists of locally constant functions on $\GL_n(F)$ with compact support when 
$F$ is non-Archimedean. 

By \cite[Theorem 5.4.3]{AG2008}, we can identify $\sS(\GL_n(F))$ with a subspace of $\sS(M_n(F))$ consisting of functions that
vanish, together with all their derivatives, on $M_n(F)\smallsetminus\GL_n(F)$ when $F$ is Archimedean. When $F$ is non-Archimedean, 
on the other hand, we can identify $\sS(\GL_n(F))$ with a subspace of $\sS(M_n(F))$ consisting of functions whose support 
is contained in $\GL_n(F)$. 

Let $\phi_1\in\sS(\GL_n(F))\subset\sS(M_n(F))$ and $\phi_2\in\sS\left(M_{n,1}(F)\right)$, and define $\varphi=\phi_1\otimes\phi_2\in\sS$ 
by
\[
\varphi(h,y)=\phi_1(h)\phi_2(y)
\quad\text{for $h\in M_n(F)$ and $y\in M_{n,1}(F)$.}
\]
We now compute $\mathcal{W}_{\varphi\otimes W}(I_{n+1})$, which is given by the following integral:
\[
\int_{N_n(F)\backslash\GL_n(F)} W(h)\int_{N_{n+1}(F)}|\det(h)|_F^{-(n+1)/2}\varphi(h^{-1}\varepsilon_nu)\psi^{-1}_{n+1}(u)dudh. 
\]
By writing
\[
u=\begin{pmatrix}v&y\\0&1\end{pmatrix}
\quad\text{with $v\in N_n(F)$ and $y\in M_{n,1}(F)$,}
\]
and noting that 
\[
\psi_{n+1}(u)=\psi_n(v)\psi(e_ny)
\quad\text{where $e_n:=(0,..., 0,1)\in M_{1,n}(F)$},
\]
the above integral becomes 
\[
\int_{N_n(F)\backslash\GL_n(F)} W(h)\phi'_2(e_nh)
\int_{N_{n}(F)}\phi'_1\left( v^{-1}h\right)\psi^{-1}_{n}(v)dv
dh,
\]
where we have defined
\[
\phi'_1\in\sS(\GL_n(F))
\quad\text{and}\quad
\phi'_2\in\sS\left(M_{1,n}(F)\right)
\]
by
\[
\phi_1'(h)=|\det(h)|_F^{-(n-1)/2}\phi_1\left(h^{-1}\right)
\quad\text{for $h\in\GL_n(F)$},
\]
and 
\[
\phi'_2(y)
=
\int_{M_{n,1}(F)}\phi_2(x)\psi^{-1}(y x)dx
\quad\text{for $y\in M_{1,n}(F)$}. 
\]

To proceed, we change the variable $v^{-1}\mapsto v$, and observe that  
\[
\psi_n(v)W(h)=W(vh)\quad\text{and}\quad \phi'_2(e_nh)=\phi'_2(e_nvh).
\]
It follows that
\begin{align*}
&\int_{N_n(F)\backslash\GL_n(F)} W(h)\phi'_2(e_nh)
\int_{N_{n}(F)}\phi'_1\left( vh\right)\psi_{n}(v)dvdh\\
&\quad=
\int_{N_n(F)\backslash\GL_n(F)}\int_{N_n(F)}
W(vh)\phi'_1(vh)\phi'_2(e_nvh)dvdh\\
&\quad\quad=
\int_{\GL_n(F)}W(h)\phi'_1(h)\phi'_2(e_nh)dh
=
0,
\end{align*}
for every $\phi_1\in\sS(\GL_n(F))$ and $\phi_2\in\sS(M_{n,1}(F))$. Since the maps $\phi_1\mapsto\phi'_1$ and $\phi_2\mapsto\phi'_2$ 
define automorphism of $\sS(\GL_n(F))$ and isomorphism from $\sS(M_{n,1}(F))$ to $\sS(M_{1,n}(F))$, respectively, we conclude that 
\[
W(h)=0\quad\text{for every $h\in\GL_n(F)$, i.e. $W=0$.}
\]
Indeed, we may assume that $\phi'_2$ vanishes nowhere, then since the space 
$\sS(\GL_n(F))$ contains approximate identities, the result follows.
\end{proof}

Now we come to the first main result of this article:

\begin{thm}\label{T:image}
Assume that $\pi$ is generic and that $\pi\times 1$ is irreducible. Then  $\sW(\pi\times 1,\psi_{n+1})=\sV(\pi,\psi_{n+1})$. 
\end{thm}

\begin{proof}
The result \cite[Theorem 2.1]{Howe1989} of Howe implies that the big theta lift $\Theta(\pi^\vee)$ is of finite length; thus it admits a composition series
\[
\{0\}=\sV_0\subset\sV_1\subset\cdots\subset\sV_M=\Theta(\pi^\vee).
\]
By the aforementioned result of Howe and Corollary \ref{C:theta}, the top quotient satisfies:
\[
\sV_M/\sV_{M-1}\cong\theta(\pi^\vee)=\pi\times 1.
\]
Note that $\pi\times 1$ is generic because $\pi$ itself is generic.
On the other hand, by Proposition \ref{P:GZ}, we have
\[
\dim_\C{\rm Hom}_{N_{n+1}(F)}\left(\Theta(\pi^\vee),\psi_{n+1}\right)=1.
\]
From these, we conclude that the subquotients $\sV_{\ell}/\sV_{\ell-1}$ for $1\le \ell\le M-1$ must be non-generic.

By the definition of $\Theta(\pi^\vee)$ and \eqref{E:Whittaker equivariant}, there exists a surjective 
$\GL_{n+1}(F)$-intertwining map
\[
\Xi: \Theta(\pi^\vee)\relbar\joinrel\twoheadrightarrow\sV(\pi,\psi_{n+1})
\subset
\mathcal{C}^\infty(N_{n+1}(F)\backslash\GL_{n+1}(F),\psi_{n+1})^{\rm mg}.
\]
This map $\Xi$ gives rise to a filtration of $\GL_{n+1}(F)$-representation
\[
0=\Xi(\sV_0)\subset\Xi(\sV_1)\subset\cdots\subset\Xi(\sV_{M-1})\subset\Xi(\sV_M)
=
\sV(\pi,\psi_{n+1}).
\]
Since $\Xi(\sV_\ell)/\Xi(\sV_{\ell-1})$ is a quotient of the irreducible representation $\sV_\ell/\sV_{\ell-1}$ for each $1\le \ell\le M$,
it is either $0$ or isomorphic to $\sV_\ell/\sV_{\ell-1}$. In particular, 
\[
\Xi(\sV_1)\subset\mathcal{C}^\infty(N_{n+1}(F)\backslash\GL_{n+1}(F),\psi_{n+1})^{\rm mg}
\] 
must be $0$ since $\sV_1$ is non-generic. Consequently, $\Xi(\sV_2)$ is either $0$ or isomorphic to $\sV_2/\sV_1$. 
Applying the same reasoning, we deduce that 
$\Xi(\sV_2)=0$. Repeating this argument, we find that $\Xi(\sV_{M-1})=0$. Now, since $\Xi(\sV_M)=\sV(\pi,\psi_{n+1})$ is 
non-zero by Lemma \ref{L:image}, it follows that
\[
\pi\times 1\cong\sV_M/\sV_{M-1}\cong\Xi(\sV_M)=\sV(\pi,\psi_{n+1}).
\]
This completes the proof.
\end{proof}

\section{Propagation formula}\label{S:propagation}
In this section, we assume that $F$ is Archimedean. Our goal is to derive ``propagation formulas'' for the Whittaker functions in the minimal $K$-type, using the theta lifting
discussed in the previous section. We retain the notation \eqref{E:case m=n+1}.

\subsection{Minimal $K$-type Whittaker functions}
Let $\pi$ be an irreducible generic representation of $\GL_n(F)$. We assume that
\[
\Pi:=\pi\times 1
\] 
is irreducible, so that $\Pi$ is an irreducible generic representation of $\GL_{n+1}(F)$

Let $\tau$ (resp. $\Upsilon$) be the minimal $K$-type of $\pi$ (resp. $\Pi$), which occurs with multiplicity one.
It follows that the subspace
\begin{equation}\label{E:Whittaker newform for GL_n}
\left(\mathscr{W}(\pi,\psi_n)\otimes \tau^\vee\right)^{K_n}=\mathbb{C}\, W_{\pi}
\end{equation}
of $K_n$-fixed vectors is one-dimensional. We can regard $W_{\pi}$ as a $\tau^\vee$-valued function on $\GL_n(F)$ 
satisfying:
\begin{itemize}
\item
$W_{\pi}(gk)=\tau^\vee(k)^{-1}W_{\pi}(g)$ for $(g,k)\in \GL_n(F)\times K_n$; 
\item
for each $v\in\tau$, the $\mathbb{C}$-valued function $g\mapsto\langle v, W_{\pi}(g)\rangle_\tau$ belongs to 
$\mathscr{W}(\pi,\psi_n)[\tau]$, the $\tau$-isotypic part of $\sW(\pi,\psi_n)$.
\end{itemize}
Here $\langle\cdot,\cdot\rangle_\tau$ denotes the natural $K_n$-invariant bilinear pairing between $\tau$ and $\tau^\vee$. Note that
$W_{\pi}$ is uniquely characterized by these two properties. 

Similarly, we have
\[
\left(\mathscr{W}(\Pi,\psi_{n+1})\otimes \Upsilon^\vee\right)^{K_{n+1}}=\mathbb{C}\, W_{\Pi}
\]
with $W_{\Pi}$ satisfying analogous properties. By the Iwasawa decomposition, it is clear that 
$W_{\pi}$ (resp. $W_{\Pi}$) is completely determined by its values on $T^+_n$ (resp. $T^+_{n+1}$).

\subsection{Joint harmonics}
We aim to express $W_{\Pi}|_{T^+_{n+1}}$ in terms of $W_{\pi}|_{T^+_n}$ using the equation \eqref{E:V}. 
The idea is to select a suitable function $\varphi_{\tau,\Upsilon}$ from the subspace
\[
\left(\mathscr{S}\otimes\left(\tau\boxtimes\Upsilon^\vee\right)\right)^{K_n\times K_{n+1}}
\]
such that the integral analogous to \eqref{E:V} with $W=W_{\pi}$ and $\varphi=\varphi_{\tau,\Upsilon}$ is non-zero.
As above, $\varphi_{\tau,\Upsilon}$ can be regarded as a $\tau\otimes\Upsilon^\vee$-valued function on $M_{n,n+1}(F)$, 
whose coefficients belong to $\mathscr{S}$, and it satisfies the identity
\begin{equation}\label{E:equivariant}
\omega(k,k')\varphi_{\tau,\Upsilon}=\left(\tau\boxtimes\Upsilon^\vee\right)(k,k')^{-1}\varphi_{\tau,\Upsilon}
\quad
\text{for $(k,k')\in K_n\times K_{n+1}$.}
\end{equation}
Here, $\omega(k,k')$ acts on the coefficients of $\varphi_{\tau,\Upsilon}$. 

To this end, let $\sH\subset\sS$ denote the subspace of joint harmonics in the sense of Howe 
(\cite[Section 7 (b)]{Howe1989}). This is a $(K_n \times K_{n+1})$-invariant subspace which admits a multiplicity-free 
decomposition
\[
\sH\cong\bigoplus_i\sH_{\sigma_i,\Sigma_i}\cong\bigoplus_i\sigma_i\boxtimes\Sigma_i.
\]
Moreover, for each $i$, $\sigma_i$ and $\Sigma_i$ determine each other, and they share the same degree defined by Howe.

Since $\pi$ and $\Pi$ are irreducible and generic, they are isomorphic to full induced representations by a result of Vogan
(\cite{Vogan1978}). Consequently, the highest weight of $\tau$ and $\Upsilon$ can be describe explicitly 
(\cite[Proposition 4.3]{Lin2018}, \cite[Proposition 4.1]{Chen2019}, \cite[Remark 6.8]{Humphries2025}).
Now the key observation is that $\tau^\vee\boxtimes\Upsilon$ occurs in $\sH$ (\cite[Proposition 1.4]{AdamsBarbasch1995}, 
\cite[Proposition 7.5]{Adams2007}). Therefore, the space 
\[
\left(\sH\otimes\left(\tau\boxtimes\Upsilon^\vee\right)\right)^{K_n\times K_{n+1}}
\]
is one-dimensional, and the distinguished function $\varphi_{\tau,\Upsilon}$ is chosen from a basis of this space.

\subsection{The formulas}
At this point, we introduce the $\Upsilon^\vee$-valued function $\mathcal{W}_{\varphi_{\tau,\Upsilon}\otimes W_{\pi}}$ on $\GL_{n+1}(F)$ 
defined by the integral:
\[
\mathcal{W}_{\varphi_{\tau,\Upsilon}\otimes W_{\pi}}(g)
=
\int_{N_n(F)\backslash\GL_n(F)}\int_{N_{n+1}(F)}
\langle\omega(h,ug)\varphi_{\tau,\Upsilon}(\varepsilon_n),W_{\pi}(h)\rangle_\tau\psi^{-1}_{n+1}(u)\,du\,dh.
\]
Since the coefficients are given by the integral \eqref{E:V}, which converges absolutely, the 
function $\mathcal{W}_{\varphi_{\tau,\Upsilon}\otimes W_{\pi}}$ is well-defined. Moreover, we have the following, which is the 
main result of this article.

\begin{thm}\label{T:main}
The function $\mathcal{W}_{\varphi_{\tau,\Upsilon}\otimes W_{\pi}}$ is non-zero, and belongs to 
\[
\left(\mathscr{W}(\Pi,\psi_{n+1})\otimes \Upsilon^\vee\right)^{K_{n+1}}. 
\]
Hence, we can identify
\[
W_{\Pi}=\mathcal{W}_{\varphi_{\tau,\Upsilon}\otimes W_{\pi}}.
\] 
Moreover, we have the following propagation formula:
\begin{align}\label{E:propagation formula}
\begin{split}
W_{\Pi}(g)
=
c|\det(g)|_F^{n/2}
&\int_{T^+_n}\int_{N_{n+1}(F)}\,du\,dt\\
&\quad\langle \varphi_{\tau,\Upsilon}(t^{-1}\varepsilon_n ug), W_{\pi}(t)\rangle_\tau
\psi^{-1}_{n+1}(u)
\delta^{-1}_{B_n(F)}(t)
|\det(t)|_F^{-(n+1)/2},
\end{split}
\end{align}
where $c={\rm Vol}(K_n, dk)$, and $\delta_{B_n(F)}$ is the modulus function of $B_n(F)$.
\end{thm}

\begin{proof}
We have noticed that the coefficients of $\mathcal{W}_{\varphi_{\tau,\Upsilon}\otimes W_{\pi}}$ are defined by the integral
\eqref{E:V}. It then follows from Theorem \ref{T:image} that these coefficients, which are functions on $\GL_{n+1}(F)$, are 
contained in $\mathscr{W}(\Pi,\psi_{n+1})$.  Moreover, it is straightforward to verify that 
$\mathcal{W}_{\varphi_{\tau,\Upsilon}\otimes W_{\pi}}$ satisfies 
\[
\mathcal{W}_{\varphi_{\tau,\Upsilon}\otimes W_{\pi}}(gk)
=
\Upsilon^\vee(k)^{-1}\mathcal{W}_{\varphi_{\tau,\Upsilon}\otimes W_{\pi}}(g)
\quad\text{for $(k,g)\in K_{n+1}\times\GL_{n+1}(F)$},
\]
using \eqref{E:equivariant}. Therefore, we conclude that
\[
\mathcal{W}_{\varphi_{\tau,\Upsilon}\otimes W_{\pi}}
\in
\left(\mathscr{W}(\Pi,\psi_{n+1})\otimes \Upsilon^\vee\right)^{K_{n+1}}.
\]

To show that $\mathcal{W}_{\varphi_{\tau,\Upsilon}\otimes W_{\pi}}$ is non-zero,  note that the minimal $K$-type $\tau^\vee$ of 
$\pi^\vee$ also has the minimal degree in the sense of Howe. It then follows from the fact that $\tau^\vee\boxtimes\Upsilon$
is contained in the joint harmonic $\sH$, and from Howe's proof,  that the maximal $\pi^\vee$-isotypic quotient of 
$\omega$, which is a representation of $\GL_n(F)\times\GL_{n+1}(F)$, is generated by $\sH_{\tau^\vee,\Upsilon}$. 

Consequently, the unique $\left(\GL_n(F)\times\GL_{n+1}(F)\right)$-intertwining map
\[
U:\omega\relbar\joinrel\twoheadrightarrow\sW(\pi^\vee,\psi_n)\,\widehat{\boxtimes}\,\sW\left(\Pi,\psi_{n+1}\right)
\]
which factors through the maximal $\pi^\vee$-isotypic quotient of $\omega$, is non-zero when restricted to the subspace 
$\sH_{\tau,^\vee,\Upsilon}$. By Theorem \ref{T:image} and the fact that the maps $U,V$ are related by 
\[
\mathcal{W}_{\varphi\otimes W}=\langle W, U(\varphi)\rangle_\pi,
\]
where $\langle\cdot,\cdot\rangle_\pi$ is the $\GL_n(F)$-invariant bilinear pairing between $\sW(\pi,\psi_n)$ and
$\sW(\pi^\vee,\psi_n)$, we conclude that the coefficients of $\mathcal{W}_{\varphi_{\pi,\Pi}\otimes W_{\pi}}$, which are
 elements in $\sW(\Pi,\psi_{n+1})[\Upsilon]$, are non-zero. This proves the first assertion.

It remains to justify \eqref{E:propagation formula}. By the Iwasawa decomposition $\GL_n(F)=N_n(F)T^+_nK_n$, and 
\eqref{E:equivariant}, we obtain
\begin{align*}
W_{\Pi}(g)
=
|\det(g)|_F^{n/2}
&\int_{N_n(F)\backslash\GL_n(F)}\int_{N_{n+1}(F)}\,du\,dh\\
&\quad\langle  \varphi_{\pi,\Pi}(h^{-1}\varepsilon_n ug), W_{\pi}(h)\rangle_\tau
\psi^{-1}_{n+1}(u)|\det(h)|_F^{-(n+1)/2}\\
=
|\det(g)|_F^{n/2}
&\int_{T^+_n}\int_{N_{n+1}(F)}\int_{K_n}\,dk\,du\,dt\\
&\quad\langle  \tau(k)^{-1}\varphi_{\pi,\Pi}(t^{-1}\varepsilon_n ug), \tau^\vee(k)^{-1}W_{\pi}(t)\rangle_\tau
\psi^{-1}_{n+1}(u)
\delta^{-1}_{B_n(F)}(t)
|\det(t)|_F^{-(n+1)/2}\\
=
|\det(g)|_F^{n/2}
&\int_{T^+_n}\int_{N_{n+1}(F)} c\,du\,dt\\
&\quad\langle  \varphi_{\tau,\Upsilon}(t^{-1}\varepsilon_n ug), W_{\pi}(t)\rangle_\tau
\psi^{-1}_{n+1}(u)
\delta^{-1}_{B_n(F)}(t)
|\det(t)|_F^{-\frac{n+1}{2}}.
\end{align*}
This completes the proof.
\end{proof}

\begin{rmk}\noindent
\begin{itemize}
\item[(1)]
By the formula \eqref{E:propagation formula}, the Whittaker function $W_{\Pi}$ can be computed inductively. 
The problem is thus reduced to finding the test function $\varphi_{\tau,\Upsilon}$, and explicitly evaluating the 
integral. We note that the space of joint harmonic can be explicitly defined in terms of the Fock model of $\omega$. 
This fact may assist in developing a systematic method for determining $\varphi_{\tau,\Upsilon}$. 
\item[(2)]
In practice, for a given specific $\pi$, one can simply choose a test function $\varphi_{\tau,\Upsilon}$ satisfying \eqref{E:equivariant}, and 
compute the integral \eqref{E:propagation formula} directly. 
The conclusion of Theorem \ref{T:main} still holds provide the integral is non-zero.
\item[(3)]
In this article, we assume that $\tau$ and $\Upsilon$ are both of minimal $K$-types. However, it is clear that 
\eqref{E:propagation formula} holds for arbitrary $K$-types, provided that $W_\pi$ and $\varphi_{\tau,\Upsilon}$ satisfy 
\eqref{E:Whittaker newform for GL_n} and \eqref{E:equivariant}, respectively. We note, however, that an arbitrary $K$-type may not 
occur with multiplicity one, so in such cases, one must also choose $W_\pi$ appropriately.
\end{itemize}
\end{rmk}

\section{Examples}\label{S:examples}

In this section, we give examples of explicit formulas for Whittaker functions using our results Theorems \ref{T:image} and \ref{T:main}.
In the Archimedean case, these formulas describe Whittaker functions in the minimal \( K \)-type
of certain generic representations, while in the non-Archimedean case,
our results recover Shintani's formula for spherical Whittaker functions.


\subsection{Shintani's formula}\label{SS:Shintani}

In this subsection, assume $F$ is non-Archimedean. We give a different proof of Shintani's formula \cite{Shintani1976} for spherical Whittaker functions based on Theorem \ref{T:image}.

Let $\pi$ be an irreducible generic unramified representation of $\GL_n(F)$. Then we have
\[
\pi = \chi_1 \times \cdots \times \chi_n
\]
for some unramified characters $\chi_1,...,\chi_n$ of $F^\times$. 
The Satake parameters of $\pi$ are the numbers $\alpha_i:=\chi_i(\varpi)$ for $1 \leq i \leq n$.
For a dominant integral weight $\lambda = (\lambda_1,...,\lambda_n)$ for $\GL_n$, let $\varpi^\lambda \in T_n^+$ defined by
\[
\varpi^\lambda:={\rm diag}(\varpi^{\lambda_1},...,\varpi^{\lambda_n}).
\]
Let $W_\pi$ be a (spherical) Whittaker function on $\GL_n(F)$ with respect to $\psi_n$ defined so that
\[
W_\pi(\varpi^\lambda k) = \delta_{B_n(F)}^{1/2}(\varpi^\lambda) \frac{\det(\alpha_i^{\lambda_j+n-j})_{1 \leq i,j \leq n}}{\det(\alpha_i^{n-j})_{1 \leq i,j \leq n}} \quad \text{for $\varpi^\lambda \in T_n^+$ and $k \in K_n$}.
\]

\begin{thm}[Shintani]
We have $W_\pi \in \sW(\pi,\psi_n)$.
\end{thm}

\begin{proof}
We prove by induction on $n$. The formula trivially holds for $n=1$. 
Suppose the formula holds for all irreducible unramified generic representations of $\GL_n(F)$.
Let 
\[
\Pi = \chi_1 \times \cdots \times \chi_{n+1}
\]
be an irreducible unramified generic representation of $\GL_{n+1}(F)$ for some unramified characters $\chi_1,...,\chi_{n+1}$ of $F^\times$.
By replacing $\Pi$ by $\Pi \otimes \chi_{n+1}^{-1}$, we may assume $\chi_{n+1} = 1$.
In this case, $\Pi = \pi \times 1$, where 
\[
\pi = \chi_1 \times \cdots \times \chi_n.
\]
By induction hypothesis, $W_\pi$ is an unramified Whittaker function of $\pi$ with respect to $\psi_n$.
Let $\varphi := \mathbb{I}_{M_{n,n+1}(\frak o)}$ be the characteristic function of $M_{n,n+1}(\frak o)$. It is clear that $\varphi$ is invariant under $K_n \times K_{n+1}$. Thus $\mathcal{W}_{\varphi \otimes W_\pi}$ is invariant by $K_{n+1}$.
Let $\lambda$ be a dominant integral weight for $\GL_{n+1}$.
Then we have
\[
\mathcal{W}_{\varphi \otimes W_\pi}(\varpi^{\lambda}) = \sum_{\mu} \Psi(\omega_{n,n+1}(1,\varpi^\lambda)\varphi;\varpi^\mu)W_\pi(\varpi^{\mu})\delta_{B_n(F)}^{-1}(\varpi^\mu),
\]
where $\mu$ runs through the dominant integral weights for $\GL_n$.
By a direct computation,
\begin{align*}
\Psi(\omega_{n,n+1}(1,\varpi^\lambda)\varphi;\varpi^\mu) &= |\det(\varpi^\mu)|_F^{-(n+1)/2}|\det(\varpi^\lambda)|_F^{n/2} \prod_{i=1}^n \mathbb{I}_{\frak{o}}(\varpi^{-\mu_i+\lambda_i}) \\
&\quad\times\prod_{1 \leq i<j\leq n+1}\int_F du_{ij}\, \mathbb{I}_{\frak{o}}(\varpi^{-\mu_i+\lambda_j}u_{ij})\psi^{-1}(u_{12}+u_{23}+\cdots+u_{n,n+1})\\
&= \delta_{B_n(F)}^{1/2}(\varpi^\mu)\delta_{B_{n+1}(F)}^{1/2}(\varpi^\lambda)\prod_{i=1}^n \mathbb{I}_{\frak{o}}(\varpi^{-\mu_i+\lambda_i})\mathbb{I}_{\frak{o}}(\varpi^{\mu_i-\lambda_{i+1}}).
\end{align*}
The above product is nonzero if and only if
\[
\lambda_i \ge \mu_i \ge \lambda_{i+1}
\quad\text{for all } 1 \le i \le n,
\]
in which case we write \(\lambda \succ \mu\).
Therefore,
\begin{align}\label{E:Shintani 1}
\mathcal{W}_{\varphi \otimes W_\pi}(\varpi^{\lambda}) = \delta_{B_{n+1}(F)}^{1/2}(\varpi^\lambda) \sum_{\lambda \succ \mu} \delta_{B_n(F)}^{-1/2}(\varpi^\mu)W_\pi(\varpi^\mu).
\end{align}
Let $V_\lambda$ be the irreducible algebraic representation of $\GL_{n+1}(\C)$ of highest weight $\lambda$, and $\chi_\lambda$ be its character. It is well-known that
\[
\chi_\lambda({\rm diag}(t_1,...,t_{n+1})) = \frac{\det(t_i^{\lambda_j+n+1-j})}{\det(t_i^{n+1-j})}\quad \text{for $t_1,...,t_{n+1} \in \C^\times$}.
\]
Similar formula holds for irreducible algebraic representations of $\GL_n(\C)$. We regard $\GL_n(\C)$ as a subgroup of $\GL_{n+1}(\C)$ via the embedding $g \mapsto {\rm diag}(g,1)$. By the branching law for the restriction of representations from $\GL_{n+1}(\C)$ to $\GL_n(\C)$, we have
\[
V_\lambda \vert_{\GL_n(\C)} = \bigoplus_{\lambda \succ \mu}V_\mu.
\]
In particular, this implies that
\begin{align}\label{E:Shintani 2}
\chi_\lambda({\rm diag}(t_1,...,t_n,1)) = \sum_{ \lambda \succ \mu} \chi_\mu ({\rm diag}(t_1,...,t_n)) \quad \text{for $t_1,...,t_{n} \in \C^\times$}.
\end{align}
We then conclude from (\ref{E:Shintani 1}) and (\ref{E:Shintani 2}) that $\mathcal{W}_{\varphi \otimes W_\pi} = W_\Pi$.
This completes the proof.
\end{proof}

\subsection{Miyazaki's formula}\label{SS:Miyazaki}
In this subsection, let \( F = \mathbb{R} \).
By explicating the propagation formula in Theorem~\ref{T:main}, we obtain explicit Whittaker functions for generalized principal series representations of \( \GL_3(\mathbb{R}) \).
These formulas were originally derived by Miyazaki~\cite{Miyazaki2009}
(see also~\cite{HIM2012}) via an explicit solution of the partial differential equations satisfied by Whittaker functions.

It is more convenient to work with $K_2^\circ={\rm SO}_2(\R)$. 
For $\theta \in \R$, let 
\[
k_\theta := \bp \cos\theta & \sin\theta \\ -\sin\theta & \cos\theta \ep \in K_2^\circ.
\]
For $\kappa \geq 2$, we denote by $D_\kappa$ the discrete series representation of ${\rm GL}_2(\R)$ with minimal $K_2^\circ$-type 
$\tau_\kappa$ which is a character defined by $\tau_\kappa(k_\theta) = e^{\i\,\kappa\theta}$. 
Besides the principal series representations, an irreducible generic representation of $\GL_3(\R)$ is of the form
\[
\Pi = (D_\kappa\otimes\chi_1) \times \chi_2
\]
for some $\kappa \geq 2$ and characters $\chi_1,\chi_2$ of $\R^\times$. 
Let $\Upsilon_\kappa$ be the minimal $K_3$-type of $\Pi$. A model of $\Upsilon_\kappa$ is recalled as follows: 
Let $\C[X,Y,Z]_\kappa$ be the space of homogeneous complex polynomials of degree $\kappa$ in variables $X,Y,Z$. 
Then $K_3$ acts on $\C[X,Y,Z]_\kappa$ by
\[
(k\cdot P)(X,Y,Z) := P((X,Y,Z)k) \quad
\text{for }k \in K_3,\, P \in \C[X,Y,Z]_\kappa.
\]
It has an invariant subspace consisting of multiples of $X^2+Y^2+Z^2$. Then $\Upsilon_\kappa$ is the quotient of $\C[X,Y,Z]_\kappa$ 
by this invariant subspace. We denote by $\overline{X}, \overline{Y}, \overline{Z}$ the image of $X,Y,Z$ in the quotient. 
By replacing $\Pi$ by $\Pi \otimes \chi_2^{-1}$, we may assume $\chi_1 = |\mbox{ }|_\R^{\sf w}$ for some ${\sf w} \in\C$ and $\chi_2 = 1$.
We use the propagation formula \eqref{E:propagation formula} to give an explicit formula for the Whittaker functions of 
$\Pi$ in the minimal type. More precisely, we apply it to 
\[
\pi = D_\kappa \otimes |\mbox{ }|_\R^{\sf w},\quad \tau = \tau_\kappa,\quad \Upsilon = \Upsilon_\kappa.
\]
Note that $\Upsilon \cong \Upsilon^\vee$.
The distinguished Bruhat-Schwartz function 
$\varphi_{\pi,\Pi} \in \mathscr{S}(M_{2,3}(\R)) \otimes \left(\tau\boxtimes\Upsilon^\vee\right)$ is defined by
\begin{align*}
\varphi_{\pi,\Pi}(x) & 
= 
\left( \bp {{\i}} & 1 \ep\cdot x\cdot {}^t\!\bp \overline{X} & \overline{Y} & \overline{Z}\ep \right)^\kappa e^{-\pi{\rm tr}(x{}^t\!x)}\\
&= 
\left( (x_{21}+\i\,x_{11})\overline{X} + (x_{22}+\i\,x_{12})\overline{Y} + (x_{23}+\i\,x_{13})\overline{Z} \right)^\kappa e^{-\pi{\rm tr}(x{}^t\!x)}.
\end{align*}
It is clear from definition that $\varphi_{\tau,\Upsilon}$ satisfies \eqref{E:equivariant}. 

We normalize the Whittaker function $W_{\pi}$ in \eqref{E:Whittaker newform for GL_n} so that $W_{\pi} (1) = e^{-2\pi}$.
Then we have
\[
W_{\pi}({\rm diag}(t_1t_2,t_2)k_\theta) = e^{\i\,\kappa\theta}t_1^{{\kappa/2}+{\sf w}}t_2^{2{\sf w}}e^{-2\pi t_1}\mathbb{I}_{\R_+}(t_1) 
\quad \mbox{for }t_1,t_2 \in \R^\times,\,k_\theta \in K_2^\circ.
\]

For $a_1,a_2,t_1,t_2 \in \R_+$ and $u=(u_{i,j}) \in N_3(\R)$, we have
\begin{align*}
& \varphi_{\pi,\Pi}\left({\rm diag}(t_1t_2,t_2)^{-1}\varepsilon_2 u \,{\rm diag}(a_1a_2,a_2,1)\right)\\
&= 
\left( \i \,t_1^{-1}t_2^{-1}a_1a_2 \overline{X} 
+ (t_2^{-1}a_2+\i\,t_1^{-1}t_2^{-1}a_2u_{12}) \overline{Y} 
+ (t_2^{-1}u_{23}+\i\,t_1^{-1}t_2^{-1}u_{13})  \overline{Z} \right)^\kappa\\
&\quad\quad\quad\quad\quad\quad\quad\quad\quad\quad\quad\quad\quad\quad\quad\quad\quad\quad\quad\quad \times 
e^{-\pi\left(t_1^{-2}t_2^{-2}(a_1^2a_2^2+a_2^2 u_{12}^2+u_{13}^2) + t_2^{-2}(a_2^2+u_{23}^2) \right)}. 
\end{align*}

\begin{lm}\label{L:Fourier transform 1}
Let $a_1,a_2,t_1,t_2 \in \R_+$ and $n,m$ be non-negative integers. We have
\begin{align*}
    &\int_{N_3(\R)} (t_2^{-1}a_2+\i\,t_1^{-1}t_2^{-1}a_2u_{12})^{n}(t_2^{-1}u_{23}+\i\,t_1^{-1}t_2^{-1}u_{13})^{m}\\
    & \quad\quad\quad\quad\quad\quad\quad\quad\times 
    e^{-\pi\left(t_1^{-2}t_2^{-2}(a_2^2 u_{12}^2+u_{13}^2) + t_2^{-2}u_{23}^2 \right)}\psi_3^{-1}(u)\,{du}\\
    &= (4\pi)^{-n/2}(\i)^{-m}t_1^2t_2^{m+3}a_2^{-1}H_n\left(\pi^{1/2}(t_2^{-1}a_2+t_1t_2a_2^{-1})\right)e^{-\pi( t_1^2t_2^2a_2^{-2}+t_2^2 )},
\end{align*}
where $H_n$ is the $n$-th Hermite polynomial 
\[
H_n(x) = (-1)^ne^{x^2}\frac{d^n}{dx^n}(e^{-x^2}).
\]
\end{lm}

\begin{proof}
We have
\begin{align*}
&\int_{\R} (t_2^{-1}a_2+\sqrt{-1}\,t_1^{-1}t_2^{-1}a_2u_{12})^n e^{-\pi t_1^{-2}t_2^{-2}a_2^2u_{12}^2-2\pi\sqrt{-1}\,u_{12}}\,du_{12}\\
& = t_1t_2a_2^{-1}\int_{\R}(t_2^{-1}a_2+\sqrt{-1}\,u_{12})^n e^{-\pi u_{12}^2-2\pi\sqrt{-1}\,t_1t_2a_2^{-1}u_{12}}\,du_{12}\\
& = (4\pi)^{-n/2}t_1t_2a_2^{-1} H_n\left(\pi^{1/2}(t_2^{-1}a_2+t_1t_2a_2^{-1})\right) e^{-\pi t_1^{2}t_2^{2}a_2^{-2}},
\end{align*}
where the last equality follows from \cite[Lemma 7.4]{Ichino2005}.
Also 
\begin{align*}
&\int_{\R^2}(t_2^{-1}u_{23}+\i\,t_1^{-1}t_2^{-1}u_{13})^{m} e^{-\pi(t_1^{-2}t_2^{-2}u_{13}^2+t_2^{-2}u_{23}^2) 
- 2\pi\sqrt{-1}\,u_{23}}\,du_{13}\,du_{23}\\
& = t_1t_2^2\int_{\R^2}(u_{23}+\sqrt{-1}\,u_{13})^m e^{-\pi(u_{13}^2+u_{23}^2) - 2\pi\sqrt{-1}\,t_2u_{23}}\,du_{13}\,du_{23}\\
& = (\sqrt{-1})^{-m}t_1t_2^{m+2}e^{-\pi t_2^2}.
\end{align*}
This completes the proof.
\end{proof}

\begin{lm}\label{L:Mellin 1}
\noindent
\begin{itemize}
\item[(1)] Let $a \in \R^\times$, $b \in \R$, and $n \in \Z_{\geq 0}$. We have
\begin{align*}
\int_{\R_+} t^s H_n(at+b)e^{-(at+b)^2}\,d^\times t=a^{-n}\frac{\Gamma(s)}{\Gamma(s-n)}\int_{\R_+} t^{s-n} e^{-(at+b)^2}\,d^\times t.
\end{align*}
\item[(2)] We have
\begin{align*}
&\int_{\R_+^2} t_1^{s_1}t_2^{s_2}e^{-\pi(t_1+t_2)^2}\,d^\times t_1\,d^\times t_2 
= 
2^{-s_1-s_2}\pi^{(1-s_1-s_2)/2}\frac{\Gamma(s_1)\Gamma(s_2)}{\Gamma(\tfrac{s_1+s_2+1}{2})}.
\end{align*}
\end{itemize}
\end{lm}

\begin{proof}
The first assertion follows from applying \cite[17.42.2-(i)]{Table2015} to $f(t) = e^{-(at+b)^2}$.
Note that 
\[
f^{(n)}(t) = (-a)^nH_n(at+b)e^{-(at+b)^2}.
\]
For the second assertion, we have
\begin{align*}
\int_{\R_+^2} t_1^{s_1}t_2^{s_2}e^{-\pi(t_1+t_2)^2}\,d^\times t_1\,d^\times t_2 & = \int_{\R_+^2} t_1^{s_1}t_2^{s_1+s_2}e^{-\pi(t_1t_2+t_2)^2}\,d^\times t_1\,d^\times t_2\\
& = \int_{\R_+^2} \frac{t_1^{s_1}}{(1+t_1)^{s_1+s_2}}t_2^{s_1+s_2}e^{-\pi t_2^2}\,d^\times t_1\,d^\times t_2\\
& = 2^{-1}\pi^{-(s_1+s_2)/2}\Gamma(\tfrac{s_1+s_2}{2}) \int_{\R_+}\frac{t_1^{s_1}}{(1+t_1)^{s_1+s_2}}\,d^\times t_1\\
& = 2^{-s_1-s_2}\pi^{(1-s_1-s_2)/2}\frac{\Gamma(s_1)\Gamma(s_2)}{\Gamma(\tfrac{s_1+s_2+1}{2})}.
\end{align*}
Here the first and second equalities follow from the change of variables $t_1 \mapsto t_1t_2$ and $t_2 \mapsto (1+t_1)^{-1}t_2$ respectively, and the last one follows from \cite[8.380.3,\,8.384.1]{Table2015} and the duplication formula.
This completes the proof.
\end{proof}

\begin{thm}[Miyazaki]\label{T:GL_3(R)}
For $a_1,a_2 \in \R_+$, we have
\begin{align*}
&W_{\Pi}\left({\rm diag}(a_1a_2,a_2,1)\right) \\
&= 2^{-4}\sum_{n_1+n_2+n_3 = \kappa}\frac{\kappa!}{n_1!n_2!n_3!}\overline{X}^{n_1}\overline{Y}^{n_2}\overline{Z}^{n_3}(\sqrt{-1})^{n_1-n_3}a_1a_2^{1-2{\sf w}}\\
&\quad\times \int_{s_1}\frac{ds_1}{2\pi\sqrt{-1}}\int_{s_2}\frac{ds_2}{2\pi\sqrt{-1}}\, a_1^{-s_1}a_2^{-s_2}\\
&\quad\quad\quad\times \frac{\Gamma_\C(s_1+\tfrac{\kappa-1}{2}+{\sf w})\Gamma_\R(s_1+n_1)\Gamma_\C(s_2+\tfrac{\kappa-1}{2}-{\sf w})\Gamma_\R(s_2+n_3)}{\Gamma_\R(s_1+s_2+n_1+n_3)}.
\end{align*}
Here the paths of integration in $s_1$ and $s_2$ are vertical lines in the complex plane, with sufficiently large real parts to keep the poles of the integrand on their left.
\end{thm}

\begin{proof}
By \eqref{E:propagation formula} and Lemma \ref{L:Fourier transform 1}, we have
\begin{align*}
&W_{\Pi}\left({\rm diag}(a_1a_2,a_2,1)\right) \\
&= 
\sum_{n_1+n_2+n_3 
= \kappa}\frac{\kappa!}{n_1!n_2!n_3!}\overline{X}^{n_1}\overline{Y}^{n_2}\overline{Z}^{n_3}(4\pi)^{-n_2/2}(\sqrt{-1})^{n_1-n_3}a_1a_2\\
&\quad\times \int_{\R_+^2}d^\times t_1\,d^\times t_2 \,t_1^{-n_1+(\kappa-1)/2+{\sf w}}t_2^{-n_1+n_3+2{\sf w}} (a_1a_2)^{n_1}
H_{n_2}\left(\pi^{1/2}(t_2^{-1}a_2+t_1t_2a_2^{-1})\right)\\
&\quad\quad\quad\quad\quad\quad\quad\quad\quad\quad\quad\quad\quad\quad\quad\quad\quad\quad\quad\times 
e^{-\pi(t_1^{-2}t_2^{-2}a_1^2a_2^2+t_2^2)-\pi(t_2^{-1}a_2+t_1t_2a_2^{-1})^2}.
\end{align*}
Let $I_{n_1,n_2,n_3}(a_1,a_2)$ be the above integration over $\R_+^2$ for each triple $(n_1,n_2,n_3) \in \Z_{\geq 0}^3$ such that 
$n_1+n_2+n_3 = \kappa$. We compute its Mellin transform as follows:
\begin{align*}
&\int_{\R_+^2}a_1^{s_1}a_2^{s_2}I_{n_1,n_2,n_3}(a_1,a_2)\,d^\times a_1\,d^\times a_2\\
&=
 \int_{\R_+^4}d^\times a_1\,d^\times a_2\,d^\times t_1\,d^\times t_2\,a_1^{s_1+n_1}a_2^{s_2+n_2}t_1^{-n_1+(\kappa-1)/2+{\sf w}}
 t_2^{-n_1+n_3+2{\sf w}} H_{n_2}\left(\pi^{1/2}(t_2^{-1}a_2+t_1t_2a_2^{-1})\right)\\
&\quad\quad\quad\quad\quad\quad\quad\quad\quad\quad\quad\quad\quad\quad\quad
\quad\quad\quad\quad\quad\quad\quad\quad\quad\times e^{-\pi(t_1^{-2}t_2^{-2}a_1^2a_2^2+t_2^2)-\pi(t_2^{-1}a_2+t_1t_2a_2^{-1})^2}\\
&= 
\left(\int_{\R_+}a_1^{s_1+n_1}e^{-\pi a_1^2}\,d^\times a_1\right) \left(\int_{\R_+}t_2^{s_2+n_3+2{\sf w}}e^{-\pi t_2^2}\,d^\times t_2\right)\\
&\quad\times \int_{\R_+^2}a_2^{-s_1+s_2}t_1^{s_1+(\kappa-1)/2+{\sf w}}H_{n_2}\left(\pi^{1/2}(a_2+t_1a_2^{-1}) \right)
e^{-\pi(a_2+t_1a_2^{-1})^2}\,d^\times a_2\,d^\times t_1\\
&= 
2^{-2}\Gamma_\R(s_1+n_1)\Gamma_\R(s_2+n_3+2{\sf w})\cdot \pi^{-n_2/2}\Gamma(s_1+\tfrac{\kappa-1}{2}+{\sf w})
\Gamma(s_1-n_2+\tfrac{\kappa-1}{2}+{\sf w})^{-1}\\
&\quad\times 
\int_{\R_+^2}a_2^{-s_1+s_2+n_2}t_1^{s_1-n_2+(\kappa-1)/2+{\sf w}}e^{-\pi(a_2+t_1a_2^{-1})^2}\,d^\times a_2\,d^\times t_1\\
 &=
 2^{-2}\Gamma_\R(s_1+n_1)\Gamma_\R(s_2+n_3+2{\sf w})\cdot \pi^{-n_2/2}\Gamma(s_1+\tfrac{\kappa-1}{2}+{\sf w})\\
&\quad\times 2^{-s_1-s_2-n_1-n_3+1-2{\sf w}}\pi^{1-(s_1+s_2+n_1+n_3)/2-{\sf w}}
\Gamma(s_2+\tfrac{\kappa-1}{2}+{\sf w})\Gamma(\tfrac{s_1+s_2+n_1+n_3+2{\sf w}}{2})^{-1}\\
&= 
2^{-4}(4\pi)^{n_2/2}\frac{\Gamma_\C(s_1+\tfrac{\kappa-1}{2}+{\sf w})\Gamma_\R(s_1+n_1)
\Gamma_\C(s_2+\tfrac{\kappa-1}{2}+{\sf w})\Gamma_\R(s_2+n_3+2{\sf w})}{\Gamma_\R(s_1+s_2+n_1+n_3+2{\sf w})}.
\end{align*}
Here the second equality follows from the change of variables $a_1 \mapsto t_1a_2^{-1}a_1$, $a_2 \mapsto t_2a_2$, and the
third and fourth ones follow from Lemma \ref{L:Mellin 1}-(1) and (2) respectively.
The assertion then follows from the Mellin inversion formula and a change of variables $s_2 \mapsto s_2-2{\sf w}$.
This completes the proof.
\end{proof}

\subsection{}\label{SS:new cases} 
In this subsection, let $F = \C$. 
For $\nu \in \C$ and $\kappa \in \Z$, let $\chi_{[\nu,\kappa]}$ denote the character of $\C^\times$ given by
\[
\chi_{[\nu,\kappa]}(z) = |z|_\C^\nu\left( \frac{z}{\sqrt{|z|_\mathbb{C}}} \right)^\kappa.
\]
For $\underline{\nu} = (\nu_1,...,\nu_n) \in \C^n$, let
\[
|\underline{\nu}| := \nu_1+\cdots+\nu_n,\quad \tilde{\underline \nu} := (\nu_2-\nu_1,...,\nu_n-\nu_1) \in \C^{n-1}.
\]
For $\underline{\nu} = (\nu_1,...,\nu_n) \in \C^n$ and $\underline{\kappa} = (\kappa_1,...,\kappa_n) \in \Z^n$, let 
$\pi_{[\underline{\nu},\underline{\kappa}]}$ be the principal series representation of $\GL_n(\C)$ defined by
\[
\pi_{[\underline{\nu},\underline{\kappa}]} :=\chi_{[\nu_1,\kappa_1]} \times\cdots \times\chi_{[\nu_n,\kappa_n]}.
\] 
When $\underline{\kappa} = 0$, we write $\pi_{\underline{\nu}} = \pi_{[\underline{\nu},0]}$. In this case, we have the following 
Mellin--Barnes integral representation of a normalized spherical Whittaker function 
$W_{\underline{\nu}} \in {\sW}(\pi_{\underline{\nu}},\psi_n)^{K_n}$ due to Ishii and Stade \cite[Proposition 2.1]{Stade1995},
\cite[Theorem 12]{IS2007}: The formula is inductively represented. 
Let $t : \R_+^{n-1} \rightarrow T_n^+$ be a homomorhism defined by
\[
t(a):= {\rm diag}(a_1\cdots a_{n-1},a_2 \cdots a_{n-1},...,a_{n-1},1) \quad \text{for $a = {\rm diag}(a_1,...,a_{n-1})$}. 
\]
Now, define $f_{\underline{\nu}} : \R_+^{n-1} \rightarrow \C$ by
\[
f_{\underline{\nu}}(a) = \delta_{B_n(\C)}(t(a))^{-1/2}W_{\underline{\nu}}(t(a)).
\]
Then $f_{\underline{\nu}}(a) = a^{\nu_1+\nu_2}K_{\nu_1-\nu_2}(4\pi a)$ for $n=2$, and 
\begin{align}\label{E:Ishii-Stade}
\begin{split}
f_{\underline{\nu}}(a) &= \int_{\color{black}z_1,...,z_{n-2}}\frac{dz}{(2\pi\sqrt{-1})^{n-2}}\,(\mathcal{M}f_{\tilde{\underline{\nu}}})(z_1,...,z_{n-2})\\
&\quad \times \int_{\color{black}s_1,...,s_{n-1}}\frac{ds}{(2\pi\sqrt{-1})^{n-1}}\, \prod_{i=1}^{n-2}a_i^{-s_i}\Gamma_\C\left( \frac{s_i-z_{i-1}}{2}+i\nu_1 \right)\Gamma_\C\left( \frac{s_i-z_{i}}{2}+i\nu_1 \right)\\
&\quad\quad\quad\quad\times a_{n-1}^{-s_{n-1}}\Gamma_\C\left( \frac{s_{n-1}-z_{n-2}}{2}+(n-1)\nu_1 \right)\Gamma_\C\left( \frac{s_{n-1}}{2}+|\tilde{\underline\nu}|+(n-1)\nu_1 \right)
\end{split}
\end{align}
for $n \geq 3$, where we put $z_0=0$, and $\mathcal{M}f_{\tilde{\underline{\nu}}}$ is the Mellin transform of $f_{\tilde{\underline{\nu}}}$.
Here the paths of integration in $z_i$ and $s_j$ are vertical lines in the complex plane, with sufficiently large real parts to keep the poles of the integrand on their left.

The formula (\ref{E:Ishii-Stade}) of Ishii--Stade can be obtained by using the propagation formula \eqref{E:propagation formula}.
More generally, assume $\pi_{[\underline{\nu},\underline{\kappa}]}$ is irreducible and 
\begin{align}\label{E:assumption}
\underline{\kappa} = (\kappa,0,...,0),\quad \kappa \geq 0.
\end{align}
In Theorem \ref{T:GL_n(C)} below we use the propagation formula to deduce a Mellin--Barnes type integral representations for Whittaker functions of $\pi_{[\underline{\nu},\underline{\kappa}]}$ in the minimal $K$-type. 
From now on we assume (\ref{E:assumption}) holds.
Let $\rho_\kappa$ be the minimal $K$-type of $\pi_{[\underline{\nu},\underline{\kappa}]}$. Then it has $(\kappa,0,...,0)$ as its 
highest weight. The action of $\rho_\kappa$ is recalled as follows:
Let $\C[X_1,...,X_{n}]_\kappa$ be the space of homogeneous complex polynomials of degree $\kappa$ in variables 
$X_1,...,X_{n}$. We write $X = (X_1,...,X_{n})$ as row vector variable.
Then $K_{n}$ acts on $\C[X_1,...,X_{n}]_\kappa$ by 
\begin{align}\label{E:minimal type model}
(\rho_\kappa(k)P)(X) := P(Xk), \quad\text{for }k \in K_{n},\,P \in \C[X_1,...,X_{n}]_\kappa.
\end{align}
Now we use the propagation formula to define a generator
\[
W_{[\underline{\nu},\underline{\kappa}]} \in \left({\sW}(\pi_{[\underline{\nu},\underline{\kappa}]},\psi_n)\otimes \overline{\rho}_\kappa \right)^{K_n}.
\]
Let $\pi := \pi_{[\tilde{\underline{\nu}},(-\kappa,...,-\kappa)]}$ and $\Pi := \pi\times 1$. Then we have
\[
\pi_{[\underline{\nu},\underline{\kappa}]} = \Pi \otimes \chi_{[\nu_1,\kappa]}.
\]
Let $\Upsilon$ and $\tau$ be the minimal $K$-type of $\Pi$ and $\pi$ respectively.
Then 
\[
\Upsilon = \rho_\kappa\otimes {\rm det}^{-\kappa},\quad\tau={\rm det}^{-\kappa}.
\]
The distinguished Bruhat-Schwartz function 
$\varphi_{\tau,\Upsilon} \in \mathscr{S}(M_{n-1,n}(\C)) \otimes \left(\tau\boxtimes\overline{\Upsilon}\right)$ is defined by
\begin{align*}
\varphi_{\tau,\Upsilon}(x) = (-1)^{\kappa(n-1)}2^{4(n-1)}\cdot\det\bp X \\ \overline{x}\ep^\kappa e^{-2\pi{\rm tr}(x{}^t\overline{x})}.
\end{align*}
It is clear that $\varphi_{\tau,\Upsilon}$ satisfies \eqref{E:equivariant}.
We normalized the Whittaker function $W_\pi$ in \eqref{E:Whittaker newform for GL_n} by 
\[
W_\pi:=W_{\underline{\tilde{\nu}}} \otimes \chi_{[0,-\kappa]}.
\]
By Theorem \ref{T:main}, we have a generator
\[
W_{\Pi}:=\mathcal{W}_{\varphi_{\tau,\Upsilon}\otimes W_{\pi}} \in \left({\sW}(\Pi,\psi_n)\otimes \overline{\Upsilon} \right)^{K_n}.
\]
We then define
\[
W_{[\underline{\nu},\underline{\kappa}]} := W_\Pi \otimes \chi_{[\nu_1,\kappa]}.
\]

For $t = {\rm diag}(t_1,...,t_{n-1}) \in T_{n-1}^+$, $\alpha = {\rm diag}(\alpha_1,...,\alpha_{n}) \in T_{n}^+$, and $u \in N_{n}$, 
let $A_i(t,\alpha,u)$ be the $n-1$ by $n-1$ matrix obtained by deleting the $i$-th column of $t^{-1}\varepsilon_{n-1} u\alpha$. 
Then we have
\begin{align*}
\varphi_{\tau,\Upsilon}(t^{-1}\varepsilon_{n-1} u\alpha) &= (-1)^{\kappa(n-1)}2^{4(n-1)}\\
&\quad\times\sum_{\ell_1+\cdots+\ell_{n}
= \kappa} \frac{\kappa!}{\ell_1!\cdots\ell_{n}!}(-1)^{\sum_{i=1}^{n}(i+1)\ell_i}X_1^{\ell_1}\cdots X_{n}^{\ell_{n}} 
\prod_{i=1}^{n}\overline{\det(A_i(t,\alpha,u))}^{\ell_i}\\
& \quad\quad\quad\quad\quad\quad\quad\quad\quad\quad\quad\quad\quad\quad\times 
e^{-2\pi\left(\sum_{i=1}^{n-1} t_i^{-2}\alpha_i^2 + \sum_{1 \leq i < j \leq n}t_i^{-2}\alpha_j^2|u_{ij}|_\C\right)}.
\end{align*}
Recall the following well-known result on Fourier transform.

\begin{lm}\label{L:Fourier transform 2}
For $N \geq 1$ and $a \in \R$, we have
\[
\int_\C \overline{z}^N e^{-2\pi|z|_\C-4\pi\sqrt{-1}\,\Re(az)}\,dz = (-\sqrt{-1})^{N}a^N e^{-2\pi a^2}.
\]
\end{lm}

\begin{lm}\label{L:Fourier transform 3}
Let $(\ell_1,...,\ell_{n}) \in \Z_{\geq 0}^{n}$ with $\ell_1+\cdots+\ell_{n} = \kappa$, $t\in T_{n-1}^+$, and $\alpha \in T_{n}^+$. We have
\begin{align*}
&\int_{N_{n}(\C)} \prod_{i=1}^{n}\overline{\det(A_i(t,\alpha,u))}^{\ell_i}e^{-2\pi\sum_{1 \leq i < j \leq n}t_i^{-2}\alpha_j^2|u_{ij}|_\C}\psi_{n}^{-1}(u)\,du\\
& = |\det(\alpha)|_\C^{-({n-1})/2}\delta_{B_{n}(\C)}(\alpha)^{1/2}|\det(t)|_\C^{n/2}\delta_{B_{n-1}(\C)}(t)^{1/2}\\
&\quad \times\prod_{i=1}^{n-1}(-\sqrt{-1})^{\tilde{\ell}_i}t_i^{2\tilde{\ell}_i -\kappa}\alpha_i^{\kappa-\tilde{\ell}_i}\alpha_{i+1}^{-\tilde{\ell}_i}e^{-2\pi t_i^2\alpha_{i+1}^{-2}}.
\end{align*}
Here $\tilde{\ell}_i = \ell_1+\cdots+\ell_i$ for $1 \leq i \leq n-1$.
\end{lm}

\begin{proof}
The integral is a linear combination of integrals of the following form
\[
\int_{N_{n}(\C)}\prod_{1 \leq i<j \leq n} \overline{u}_{ij}^{a_{ij}}e^{-2\pi t_i^{-2}\alpha_j^2|u_{ij}|_\C}\psi_{n}^{-1}(u)\,du,\quad a_{ij} \geq 0.
\]
By Lemma \ref{L:Fourier transform 2}, the above integral is zero if $a_{ij}>0$ for some $j \geq i+2$.
Therefore, to evaluate the integral, we may replace $A_i(t,\alpha,u)$ by $A_i(t,\alpha,u')$, where
\[
u_{ij}' = \begin{cases}
u_{ij} & \mbox{ if $i \leq j+1$},\\
0 & \mbox{ otherwise}.
\end{cases}
\]
Note that
\[
\det(A_i(t,\alpha,u')) = \prod_{j=1}^{i-1}t_j^{-1}\alpha_j \prod_{j=i}^{n-1}t_j^{-1}\alpha_{j+1}u_{j,j+1}.
\]
Hence we have
\begin{align*}
&\int_{N_{n}(\C)} \prod_{i=1}^{n}\overline{\det(A_i(t,\alpha,u))}^{\ell_i}e^{-2\pi\sum_{1 \leq i < j \leq n}t_i^{-2}\alpha_j^2|u_{ij}|_\C}
\psi_{n}^{-1}(u)\,du\\
&= 
\prod_{i=1}^{n-1}(t_i^{-1}\alpha_i)^{\kappa-\tilde{\ell}_i}\int_{N_{n}(\C)}
\prod_{i=1}^{n-1} (t_i^{-1}\alpha_{i+1}\overline{u}_{i,i+1})^{\tilde{\ell}_i}e^{-2\pi\sum_{1 \leq i < j \leq n}
t_i^{-2}\alpha_j^2|u_{ij}|_\C-4\pi\sqrt{-1}\,\sum_{i=1}^{n-1} \Re(u_{i,i+1})}\,du\\
&= 
\prod_{i=1}^{n-1}(t_i^{-1}\alpha_i)^{\kappa-\tilde{\ell}_i}t_i^{2n-2i}\alpha_{i+1}^{-2i}\int_{N_{n}(\C)}
\prod_{i=1}^{n-1} \overline{u}_{i,i+1}^{\tilde{\ell}_i}e^{-2\pi\sum_{1 \leq i < j \leq n}|u_{ij}|_\C-4\pi\sqrt{-1}\,
\sum_{i=1}^{n-1} \Re(t_i\alpha_{i+1}^{-1}u_{i,i+1})}\,du\\
&= 
\prod_{i=1}^{n-1}(t_i^{-1}\alpha_i)^{\kappa-\tilde{\ell}_i}t_i^{2n-2i}\alpha_{i+1}^{-2i} 
(-\sqrt{-1})^{\tilde{\ell}_i}(t_i\alpha_{i+1}^{-1})^{\tilde{\ell}_i}e^{-2\pi t_i^2\alpha_{i+1}^{-2}}.
\end{align*}
Here the third equality follows from the change of variables $u_{ij} \mapsto t_i\alpha_j^{-1}u_{ij}$, and the fourth one follows from 
Lemma \ref{L:Fourier transform 2}. Finally, note that 
\[
\prod_{i=1}^{n-1}t_i^{2n-2i}\alpha_{i+1}^{-2i} 
= 
|\det(\alpha)|_\C^{-({n-1})/2}\delta_{B_{n}(\C)}(\alpha)^{1/2}|\det(t)|_\C^{n/2}\delta_{B_{n-1}(\C)}(t)^{1/2}.
\]
This completes the proof.
\end{proof}

For $\underline{\ell} = (\ell_1,...,\ell_n) \in \Z_{\geq 0}^n$ with $\ell_1+\cdots+\ell_n = \kappa$, we write 
$X^{\underline{\ell}} = X_1^{\ell_1}\cdots X_n^{\ell_n}$.
It is clear that the weight vectors in the minimal $K$-type of $\pi_{[\underline{\nu},\underline{\kappa}]}$ are indexed by these 
$\underline{\ell}$. In the Whittaker model of $\pi_{[\underline{\nu},\underline{\kappa}]}$, we normalize these weight vectors 
$W_{[\underline{\nu},\underline{\kappa}],\,\underline{\ell}}$ so that
\begin{align}\label{E:vector valued Whittaker}
W_{[\underline{\nu},\underline{\kappa}]} 
= \sum_{\ell_1+\cdots+\ell_n = \kappa} \frac{\kappa!}{\ell_1!\cdots\ell_{n}!}(\sqrt{-1})^{\sum_{i=1}^{n-1}\tilde{\ell}_i}X^{\underline{\ell}}
\cdot W_{[\underline{\nu},\underline{\kappa}],\,\underline{\ell}}.
\end{align}
Here $\tilde{\ell}_i = \ell_1+\cdots+\ell_i$.
For each $\underline{\ell}$ with $\ell_1+\cdots+\ell_n = \kappa$, define 
$f_{[\underline{\nu},\underline{\kappa}],\,\underline{\ell}} : \R_+^{n-1}\rightarrow \C$ by
\[
f_{[\underline{\nu},\underline{\kappa}],\,\underline{\ell}}(a) 
= 
\delta_{B_n(\C)}(t(a))^{-1/2}W_{[\underline{\nu},\underline{\kappa}],\,\underline{\ell}}(t(a)).
\]

\begin{thm}\label{T:GL_n(C)}
Let $\underline{\ell} = (\ell_1,...,\ell_n) \in \Z_{\geq 0}^n$ with $\ell_1+\cdots+\ell_n = \kappa$.
We have
\begin{align*}
&f_{[\underline{\nu},\underline{\kappa}],\,\underline{\ell}}(a_1,...,a_{n-1})\\ 
&=
\int_{\color{black}z_1,...,z_{n-2}}\frac{dz}{(2\pi\sqrt{-1})^{n-2}}\,(\mathcal{M}{\color{black}f_{\tilde{\underline \nu}}})(z_1,...,z_{n-2})\\
&\quad\times\int_{\color{black}s_1,...,s_{n-1}}\frac{ds}{(2\pi\sqrt{-1})^{n-1}}\,\prod_{i=1}^{n-2} a_i^{-s_i}
\Gamma_\C\left(\frac{s_i+\kappa-\tilde{\ell}_i-z_{i-1}}{2}+i\nu_1\right)
\Gamma_\C\left(\frac{s_i+\tilde{\ell}_i-z_i}{2}+i\nu_1\right)\\
&\quad\quad\times a_{n-1}^{-s_{n-1}}
\Gamma_\C\left(\frac{s_{n-1}+\kappa-\tilde{\ell}_{n-1}-z_{n-2}}{2}+(n-1)\nu_1\right)
\Gamma_\C\left(\frac{s_{n-1}+\tilde{\ell}_{n-1}}{2}+|\tilde{\underline\nu}|+(n-1)\nu_1\right).
\end{align*}
Here we put $z_0 =0$ and $\tilde{\ell}_i = \ell_1 + \cdots + \ell_i$ for $1 \leq i \leq n-1$.
\end{thm}

\begin{proof}
Let $a = (a_1,...,a_{n-1}) \in \R_+^{n-1}$. Put $\alpha_i = \prod_{j=i}^{n-1} a_j$ for $1 \leq i \leq {n-1}$ and $\alpha_{n}=1$.
By the propagation formula \eqref{E:propagation formula} and Lemma \ref{L:Fourier transform 3}, we have
\begin{align}\label{E:GL_n(C) pf 1}
\begin{split}
&\delta_{B_{n}(\C)}(t(a))^{-1/2}W_{\Pi}(t(a)) \\
& = (-1)^{\kappa(n-1)}2^{4(n-1)}\sum_{\ell_1+\cdots+\ell_{n}= \kappa} \frac{\kappa!}{\ell_1!\cdots\ell_{n}!}(-1)^{\sum_{i=1}^{n}(i+1)\ell_i}(-\sqrt{-1})^{\sum_{i=1}^{n-1}\tilde{\ell}_i}X^{\underline{\ell}}\\
&\quad\times \int_{\R_+^{n-1}} \prod_{i=1}^{n-1} d^\times t_i\,f_{\tilde{\underline \nu}}(t_1t_2^{-1},t_2t_3^{-1},
...,t_{n-2}t_{n-1}^{-1})t_{n-1}^{2|\tilde{\underline\nu}|} 
\prod_{i=1}^{n-1}t_i^{2\tilde{\ell}_i-\kappa}\alpha_i^{\kappa-\tilde{\ell}_i}\alpha_{i+1}^{-\tilde{\ell}_i}
e^{-2\pi(t_i\alpha_{i+1}^{-1}+t_i^{-2}\alpha_i^2)}.
\end{split}
\end{align}
Note that $(-1)^{\sum_{i=1}^{n}(i+1)\ell_i+\sum_{i=1}^{n-1}\tilde{\ell}_i} = (-1)^{\kappa(n-1)}$. 
For $z \in \C$ and $b_1,b_2 \in \R_+$, recall we have (cf.\,\cite[17.43.18]{Table2015})
\begin{align*}
&\int_0^\infty t^z e^{-2\pi(t^2b_1^2+t^{-2}b_2^2)}\,d^\times t\\
&= (b_1^{-1}b_2)^{z/2}K_{z/2}(4\pi b_1b_2) \\ 
& = 2^{-4}(b_1^{-1}b_2)^{z/2}
\int_{\color{black}s}\Gamma_\C\left(\frac{2s+z}{4}\right)\Gamma_\C\left(\frac{2s-z}{4}\right)(b_1b_2)^{-s}\,\frac{ds}{2\pi\sqrt{-1}}.
\end{align*}
By the Mellin inversion formula, we have
\[
f_{\tilde{\underline \nu}}(t_1t_2^{-1},t_2t_3^{-1},...,t_{n-2}t_{n-1}^{-1}) 
= 
\int_{\color{black}z_1,...,z_{n-2}}\frac{dz}{(2\pi\sqrt{-1})^{n-2}}\,
\prod_{i=1}^{n-2} (t_it_{i+1}^{-1})^{-z_i}(\mathcal{M}f_{\tilde{\underline \nu}})(z_1,...,z_{n-2}). 
\]
Note that 
\[
\prod_{i=1}^{n-2} (t_it_{i+1}^{-1})^{-z_i} = \prod_{i=1}^{n-2} t_i^{-z_{i}+z_{i-1}}\cdot t_{n-1}^{z_{n-2}},
\]
where we put $z_0=0$.
Therefore, in the right-hand side of (\ref{E:GL_n(C) pf 1}), the integral indexed by $(\ell_1,...,\ell_{n+1})$ is equal to 
\begin{align*}
&2^{-4(n-1)}\prod_{i=1}^{n-1}\alpha_i^{\kappa-\tilde{\ell}_i}\alpha_{i+1}^{-\tilde{\ell}_i}
\int_{\color{black}z_1,...,z_{n-2}}\frac{dz}{(2\pi\sqrt{-1})^{n-2}}\,(\mathcal{M}f_{\tilde{\underline \nu}})(z)
\int_{\color{black}s_1,...,s_{n-1}}\frac{ds}{(2\pi\sqrt{-1})^{n-1}}\,\prod_{i=1}^{n-1}a_i^{-s_i}\\
&\times \prod_{i=1}^{n-2} (\alpha_i\alpha_{i+1})^{(2\tilde{\ell}_i-\kappa-z_i+z_{i-1})/2}
\Gamma_\C\left( \frac{2s_i + 2\tilde{\ell}_i-\kappa-z_i+z_{i-1}}{4} \right)
\Gamma_\C\left( \frac{2s_i - 2\tilde{\ell}_i+\kappa+z_i-z_{i-1}}{4}\right)\\
&\quad\quad\quad\quad\quad\quad\quad\quad\quad\quad\quad\quad\quad\times 
\alpha_{n-1}^{(2\tilde{\ell}_{n-1}-\kappa+2|\tilde{\nu}|+z_{n-2})/2}
\Gamma_\C\left( \frac{2s_{n-1} + 2\tilde{\ell}_{n-1}-\kappa+2|\tilde{\underline\nu}|+z_{n-2}}{4} \right)\\
&\,\,\,\,\quad\quad\quad\quad\quad\quad\quad\quad\quad\quad\quad\quad\quad\quad\quad\quad\quad\quad\quad\quad\quad\times 
\Gamma_\C\left( \frac{2s_{n-1} - 2\tilde{\ell}_{n-1}+\kappa-2|\tilde{\underline\nu}|-z_{n-2}}{4}\right).
\end{align*}
Note that 
\[
\prod_{i=1}^{n-2} (\alpha_i\alpha_{i+1})^{-z_i+z_{i-1}}\cdot \alpha_{n-1}^{z_{n-2}} 
= 
\prod_{i=1}^{n-2}a_i^{-z_i-z_{i-1}}\cdot a_{n-1}^{-z_{n-2}}.
\]
By the change of variables
\[
s_i \longmapsto
\begin{cases}
s_i+\frac{\kappa-z_i-z_{i-1}}{2} & \mbox{ if $1 \leq i \leq n-2$},\\
s_i+\frac{\kappa+2|\tilde{\underline\nu}|-z_{i-1}}{2}, & \mbox{ if $i=n-1$}
\end{cases}
\]
we see that the above integral is equal to
\begin{align*}
&2^{-4(n-1)}\int_{\color{black}z_1,...,z_{n-2}}\frac{dz}{(2\pi\sqrt{-1})^{n-2}}\,(\mathcal{M}f_{\tilde{\underline \nu}})(z_1,...,z_{n-2})\\
&\quad\quad\quad\times
\int_{\color{black}s_1,...,s_{n-1}}\frac{ds}{(2\pi\sqrt{-1})^{n-1}}\,\prod_{i=1}^{n-2} a_i^{-s_i}
\Gamma_\C\left(\frac{s_i+\kappa-\tilde{\ell}_i-z_{i-1}}{2}\right)
\Gamma_\C\left(\frac{s_i+\tilde{\ell}_i-z_i}{2}\right)\\
&\quad\quad\quad\quad\quad\quad\quad\quad\quad\times a_{n-1}^{-s_{n-1}}
\Gamma_\C\left(\frac{s_{n-1}+\kappa-\tilde{\ell}_{n-1}-z_{n-2}}{2}\right)
\Gamma_\C\left(\frac{s_{n-1}+\tilde{\ell}_{n-1}}{2}+|\tilde{\underline\nu}|\right).
\end{align*}
Finally, by definition we have
\[
W_{[\underline{\nu},\underline{\kappa}]}(t(a)) = \prod_{i=1}^{n-1}a_i^{2i\nu_1}W_{\Pi}(t(a)).
\]
The assertion then follows from the change of variables $s_i \mapsto s_i+2i\nu_1$ for $1 \leq i \leq n-1$.
This completes the proof.
\end{proof}

\begin{rmk}
When $\underline{\ell}=(\kappa, 0,...,0)$, Ishii--Miyazaki (\cite[Theorem A.1]{IshiiMiyazaki2022}) and Humphries 
(\cite[Lemma 9.8]{Humphries2025}) also obtained propagation formulas for $f_{[\underline{\nu},\underline{\kappa}],\,\underline{\ell}}$.
\end{rmk}

\section{Asai local zeta integrals}\label{S:Asai}

In this section, we use the Mellin--Barnes type integral representation to compute the local zeta integrals introduced by Flicker 
\cite{Flicker1993} and Kable \cite{Kable2004}. Let $F/F_0$ be a quadratic extension of local fields of characteristic zero, and 
$c \in {\rm Gal}(F/F_0)$ the non-trivial automorphism. We assume that the non-trivial additive character $\psi$ of $F$ is given by $\psi = \psi_{F_0}\circ{\rm Tr}_{F/F_0}$ for some additive character $\psi_{F_0}$ of $F_0$. Let $\pi$ be an irreducible admissible generic representation of $\GL_n(F)$. 
For $\varphi \in \mathcal{S}(M_{1,n}(F_0))$ and $W \in {\sW}(\pi,{\color{black}\psi_n})$, we have the local zeta integral 
\begin{align}\label{E:local Asai zeta}
Z(s,\varphi,W) 
= 
\int_{N_n(F_0) \backslash \GL_n(F_0)} \varphi(e_ng)W\left({\rm diag}(\delta^{n-1},\delta^{n-2},...,1)g\right)|\det(g)|_{F_0}^s\,dg.
\end{align}
Here $e_n = (0,...,0,1)$ and $\delta \in F$ is a fixed element with ${\rm Tr}_{F/F_0}(\delta)=0$.
The integral converges absolutely for $\Re(s)$ sufficiently large and admits meromorphic continuation to the whole complex plane 
(cf.\,\cite[Theorem 1-(i)]{Raphael2018}, see also \cite[Appendix]{Flicker1993} and \cite[Theorem 2]{Kable2004} when $F$ is 
non-Archimedean). On the other hand, we have the Asai $L$-function of $\pi$ defined as follows: 
Let $\phi_\pi : W\!D_F \rightarrow {\rm GL}_n(\C)$ be the Langlands parameter of $\pi$. 
Denote by ${\rm As}(\phi_\pi) : W\!D_{F_0} \rightarrow {\rm GL}_{n^2}(\C)$ the multiplicative induction from $W\!D_F$ to $W\!D_{F_0}$ 
(cf.\,\cite[\S\,7]{Prasad1992}). The Asai $L$-function of $\pi$ is defined by
\[
L(s,\pi,{\rm As}):=L(s, {\rm As}(\phi_\pi)).
\]
In particular, if $\pi$ is a principal series representation
\[
\pi = \chi_1\times \cdots \times \chi_n,
\]
then we have (cf.\,\cite[Lemma 7.1-(d)]{Prasad1992})
\begin{align}\label{E:Asai local}
L(s,\pi,{\rm As}) = \prod_{i=1}^nL(s,\chi_i\vert_{F^\times})\prod_{1 \leq i < j \leq n} L(s,\chi_i\chi_j^c).
\end{align}
It is expected that $L(s,\pi,{\rm As})$ is a ``greatest common divisor" of the local zeta integrals. More precisely, we expect that there exist finitely many $\varphi_i$ and $W_i$ such that  
\[
L(s,\pi,{\rm As}) = \sum_{i} Z(s,\varphi_i,W_i).
\]
When $F$ is non-Archimedean, this is proved by Matringe in \cite[Theorem 5.3]{Matringe2011}.
When $F=\C$ and $n=2$, this is proved by the authors and Ishikawa in \cite{CCI2020}.
When $F=\C$ and $n \geq 3$, as far as the authors are aware, no relevant results have been proved in the existing literature. 
In this section, we show that the expectation holds for $\pi = \pi_{[\underline{\nu},\underline{\kappa}]}$ satisfying (\ref{E:assumption}), up to a polynomial in $s$ of degree $\lfloor \frac{\kappa}{2} \rfloor$.

Let $\pi = \pi_{[\underline{\nu},\underline{\kappa}]}$ satisfying (\ref{E:assumption}). Recall the minimal $K$-type $\rho_\kappa$ of $\pi_{[\underline{\nu},\underline{\kappa}]}$ whose model is described in (\ref{E:minimal type model}). Let $\<\cdot,\cdot\>_{\rho_\kappa}$ be a $K_n$-invariant hermitian pairing on $\rho_\kappa \times \rho_\kappa$ normalized so that
\[
\<X_n^{\kappa},X_n^\kappa\>_{\rho_\kappa} = 1.
\]
Recall in this case $K_n={\rm U}_n(\mathbb{R})$.
Define the Bruhat-Schwartz function $\varphi_\kappa \in \mathcal{S}(M_{1,n}(\R)) \otimes \rho_\kappa$ by
\[
\varphi_\kappa(x) = (x{}^t\!X)^\kappa\, e^{-\pi x {}^t\!x}.
\]
Note that we have
\begin{align}\label{E:equivariant kappa}
\varphi_\kappa(xk) = (\rho_\kappa(k)^{-1}\varphi_\kappa)(x),\quad k \in {\color{black}{\rm O}_n(\R)}.
\end{align}
Let $W_{[\underline{\nu},\underline{\kappa}]}$ be the $\overline{\rho}_\kappa$-valued Whittaker function of $\pi_{[\underline{\nu},\underline{\kappa}]}$ normalized in (\ref{E:vector valued Whittaker}).
The coefficients of $W_{[\underline{\nu},\underline{\kappa}]}$ are weight vectors in the minimal $K$-type $\rho_\kappa$.
In Theorem \ref{T:Asai local} below, we compute the local zeta integral
\begin{align*}
&Z(s,\varphi_\kappa,W_{[\underline{\nu},\underline{\kappa}]}) \\
&:= \int_{N_n(\R) \backslash \GL_n(\R)} \left\<\varphi_\kappa(e_ng),\,W_{[\underline{\nu},\underline{\kappa}]}\left({\rm diag}((\sqrt{-1})^{n-1},(\sqrt{-1})^{n-2},...,1)g\right) \right\>_{\rho_\kappa}|\det(g)|_\R^s\,dg.
\end{align*}
It is a finite sum of local zeta integrals of the form (\ref{E:local Asai zeta}).

We begin with the following result of Ishii and Stade \cite[Proposition 2.6]{IS2013} on an expression of class one
Whittaker functions (\ref{E:Ishii-Stade}).
\begin{prop}[Ishii--Stade]\label{P:IS}
Let $\underline{\mu} = (\mu_1,...,\mu_{n-1}) \in \C^{n-1}$. For any $\sigma \in \C$, we have
\begin{align*}
&(\mathcal{M}f_{\underline{\mu}})(w_1,...,w_{n-2}) \\
&= 4^{-n+2}\frac{\Gamma_\C\left( \frac{w_{n-2}}{2}+|\underline{\mu}|+\sigma\right)}{\prod_{i=1}^{n-1}\Gamma_\C(\mu_i+\sigma)}\\
&\quad\times \int_{\color{black}z_1,...,z_{n-2}}\frac{dz}{(2\pi\sqrt{-1})^{n-2}}\,(\mathcal{M}f_{\underline{\mu}})(z_1,...,z_{n-2})\prod_{i=1}^{n-2}\Gamma_\C\left( \frac{w_i-z_i}{2}\right)\Gamma_\C\left( \frac{w_{i-1}-z_i}{2}+\sigma \right).
\end{align*}
Here we put $w_0=0$.
\end{prop}

\begin{proof}
In \cite[Proposition 2.6]{IS2013}, the result is established for class one Whittaker functions over $\R$. The assertion presented 
here can be deduced from that result in conjunction with Stade's \cite[Proposition 2.1]{Stade1995}. Alternatively, the assertion can be 
proved directly by induction following the arguments in \cite{IS2013}. Ultimately, the proof reduces to the Mellin–Barnes type 
integral representation (\ref{E:Ishii-Stade}) or Theorem \ref{T:GL_n(C)}, along with Barnes' first lemma, which states that
\begin{align*}
&\int_{\color{black}z} \Gamma_\C\left( \frac{z}{2}+a \right)
\Gamma_\C\left( \frac{z}{2}+b \right)
\Gamma_\C\left( -\frac{z}{2}+c \right)
\Gamma_\C\left( -\frac{z}{2}+d \right)\,\frac{dz}{2\pi\sqrt{-1}}\\
& 
= 
4\cdot\frac{\Gamma_\C(a+c)\Gamma_\C(a+d)\Gamma_\C(b+c)\Gamma_\C(b+d)}{\Gamma_\C(a+b+c+d)}.
\end{align*}
\end{proof}

\begin{thm}\label{T:Asai local}
We have
\[
Z(s,\varphi_\kappa,W_{[\underline{\nu},\underline{\kappa}]}) = 2^{\color{black}n(n-2)}\frac{\Gamma_\R(s+2\nu_1+\kappa)}{\Gamma_\R(s+2\nu_1+\epsilon)}\cdot L(s,\pi_{[\underline{\nu},\underline{\kappa}]},{\rm As}).
\]
Here $\epsilon \in \{0,1\}$ such that $\kappa \equiv \epsilon\,({\rm mod}\,2)$.
\end{thm}

\begin{proof}
By the equivariant property of $W_{[\underline{\nu},\underline{\kappa}]}$ and (\ref{E:equivariant kappa}), we have
\[
\left\< \varphi_\kappa(xk),\, W_{[\underline{\nu},\underline{\kappa}]}(gk) \right\>_{\rho_\kappa} 
= 
\left\< \varphi_\kappa(x),\, W_{[\underline{\nu},\underline{\kappa}]}(g) \right\>_{\rho_\kappa}
\]
for all $x \in M_{1,n}(\R)$, $g \in \GL_n(\C)$, and $k \in {\rm O}_n(\R)$.
Therefore, we have
\begin{align*}
&Z(s,\varphi_\kappa,W_{[\underline{\nu},\underline{\kappa}]}) \\
& = \int_{T_n^+} \left\< \varphi_\kappa(e_nt),\, W_{[\underline{\nu},\underline{\kappa}]}({{\rm diag}((\sqrt{-1})^{n-1},(\sqrt{-1})^{n-2},...,1)}t)\right\>_{\rho_\kappa}
 |\det(t)|_\R^s\delta_{B_n(\R)}(t)^{-1}\,dt\\
& =
 \<X_n^{\kappa},X_n^\kappa\>_{\rho_\kappa}\int_{\R_+^n} \,\prod_{i=1}^n d^\times a_i\,a_n^\kappa e^{-\pi a_n^2} 
 W_{[\underline{\nu},\underline{\kappa}],\,(0,...,\kappa)}(t(a))\delta_{B_n(\C)}(t(a))^{-1/2} \prod_{i=1}^n a_i^{is}\\
& = 
\left(\int_0^\infty a_n^{ns+\kappa+2|\underline{\nu}|}e^{-\pi a_n^2}\,d^\times a_n\right) \int_{\R_+^{n-1}}\,
\prod_{i=1}^{n-1} d^\times a_i\, f_{[\underline{\nu},\underline{\kappa}],\,(0,...,\kappa)}(a_1,...,a_{n-1})
\prod_{i=1}^{n-1}a_i^{is}\\
& = 
2^{-2+(ns+\kappa+2|\underline{\nu}|)/2}\Gamma_\C\left(\frac{ns+\kappa}{2}+|\underline{\nu}|\right)
\cdot(\mathcal{M}f_{[\underline{\nu},\underline{\kappa}],\,(0,...,\kappa)})(s,...,is,...,(n-1)s).
\end{align*}
By Theorem \ref{T:GL_n(C)}, we have
\begin{align*}
&(\mathcal{M}f_{[\underline{\nu},\underline{\kappa}],\,(0,...,\kappa)})(s,...,is,...,(n-1)s)\\
& = \int_{\color{black}z_1,...,z_{n-2}}\frac{dz}{(2\pi\sqrt{-1})^{n-2}}\,(\mathcal{M}{f_{\tilde{\underline \nu}}})(z_1,...,z_{n-2})\\
&\quad\times\prod_{i=1}^{n-2} \Gamma_\C\left(\frac{is+\kappa-z_{i-1}}{2}+i\nu_1\right)\Gamma_\C\left(\frac{is-z_i}{2}+i\nu_1\right)\\
&\quad\times
\Gamma_\C\left(\frac{(n-1)s+\kappa-z_{n-2}}{2}+(n-1)\nu_1\right)
\Gamma_\C\left(\frac{(n-1)s}{2}+|\tilde{\underline\nu}|+(n-1)\nu_1\right)\\
&=
\Gamma_\C\left( \frac{s+\kappa}{2}+\nu_1 \right)\Gamma_\C\left(\frac{(n-1)s}{2}+|\tilde{\underline\nu}|+(n-1)\nu_1\right)\\
& \quad\times \int_{\color{black}z_1,...,z_{n-2}}\frac{dz}{(2\pi\sqrt{-1})^{n-2}}\,(\mathcal{M}{f_{\tilde{\underline \nu}}})(z_1,...,z_{n-2})\\
&\quad\quad\times\prod_{i=1}^{n-2} \Gamma_\C\left(\frac{is-z_i}{2}+i\nu_1\right)\Gamma_\C\left(\frac{(i+1)s
+\kappa-z_i}{2}+(i+1)\nu_1\right).
\end{align*}
In Proposition \ref{P:IS}, let 
\begin{align*}
\begin{cases}
\underline{\mu} = \tilde{\underline{\nu}},\\
w_i = is+2i\nu_1, & \mbox{ $1 \leq i \leq n-2$}\\
\sigma = s+2\nu_1+\frac{\kappa}{2},
\end{cases}
\end{align*}
we see that the above integral in $z_1,...,z_{n-2}$ is equal to 
\[
4^{n-2}\frac{\prod_{i=1}^{n-1}
\Gamma_\C\left(s+\nu_1+\nu_{i+1}+\frac{\kappa}{2}\right)}{ \Gamma_\C\left(\frac{ns+\kappa}{2}+|\underline{\nu}| \right)}\cdot 
(\mathcal{M}f_{\tilde{\underline{\nu}}})(s+2\nu_1,...,i(s+2\nu_1),...,(n-2)(s+2\nu_1)).
\]
Note that 
\[
\Gamma_\C\left(s+\nu_1+\nu_{i+1}+\frac{\kappa}{2}\right) = L(s,\chi_{[\nu_1,\kappa]}\chi_{[\nu_{i+1},0]}^c),\quad 1 \leq i \leq n-1
\]
and
\[
\Gamma_\C\left( \frac{s+\kappa}{2}+\nu_1 \right) = 2^{1-(s+2\nu_1+\kappa)/2}\Gamma_\R(s+2\nu_1+\kappa).
\]
In conclusion, we have
\begin{align}\label{E:Asai local pf 1}
\begin{split}
&(\mathcal{M}f_{[\underline{\nu},\underline{\kappa}],\,(0,...,\kappa)})(s,...,is,...,(n-1)s)\\
& = 
2^{2(n-2)+1-(s+2\nu_1+\kappa)/2}\frac{\Gamma_\R(s+2\nu_1+\kappa)}{\Gamma_\R(s+2\nu_1+\epsilon)}
\frac{\Gamma_\C\left(\frac{(n-1)(s+2\nu_1)}{2} 
+ 
|\tilde{\underline{\nu}}| \right)}{\Gamma_\C\left( \frac{ns+\kappa}{2} + |\underline{\nu}| \right)} \\
&\quad\times L(s,\chi_{[\nu_1,\kappa]}\vert_{\R^\times})\prod_{i=1}^{n-1}L(s,\chi_{[\nu_1,\kappa]}\chi_{[\nu_{i+1},0]}^c)\\
&\quad \times (\mathcal{M}f_{\tilde{\underline{\nu}}})(s+2\nu_1,...,i(s+2\nu_1),...,(n-2)(s+2\nu_1)).
\end{split}
\end{align}
By (\ref{E:Asai local}), we have
\[
L(s,\pi_{[\underline{\nu},\underline{\kappa}]},{\rm As}) = L(s,\chi_{[\nu_1,\kappa]}\vert_{\R^\times})
\prod_{i=1}^{n-1}L(s,\chi_{[\nu_1,\kappa]}\chi_{[\nu_{i+1},0]}^c) \cdot L(s+2\nu_1,\pi_{\tilde{\underline{\nu}}},{\rm As}).
\]
In particular, based on (\ref{E:Asai local pf 1}) inductively (with $\kappa=0$), for all $N \geq 2$ and $\underline{\lambda} \in \C^N$ we have
\begin{align*}
(\mathcal{M}f_{\underline{\lambda}})(s,...,is,...,(N-1)s) & = 2^{\sum_{i=1}^{N-1}\left(2(N-i-1)+1-(s+2\lambda_i)/2\right)} \frac{\Gamma_\C(\tfrac{s+2\lambda_n}{2})}{\Gamma_\R(s+2\lambda_n)} \frac{L(s,\pi_{\underline{\lambda}},{\rm As})}{\Gamma_\C\left( \frac{Ns}{2}+|\underline{\lambda}| \right)}\\
& = 2^{-Ns/2+{\color{black}N(N-2)+2-|\underline{\lambda}|}}\cdot \frac{L(s,\pi_{\underline{\lambda}},{\rm As})}{\Gamma_\C\left( \frac{Ns}{2}+|\underline{\lambda}| \right)}.
\end{align*}
The assertion then follows immediately from this formula with $N=n-1$ and $\underline{\lambda} = \tilde{\underline{\nu}}$.
This completes the proof.
\end{proof}

\subsection*{Acknowledgements}
The authors would like to thank Wan-Yu Tsai, Chen-Bo Zhu, and Binyong Sun for kindly answering their questions and for suggesting appropriate references. The authors also thank Hiraku Atobe, Miki Hirano, Taku Ishii, and Tadashi Miyazaki for their comments, and Kenichi Namikawa and Fu-Tsun Wei for their constant interest in this project. The authors particularly thank the National Center for Theoretical Sciences for its constant support and hospitality during the preparation of this paper.
This work is partially supported by the NSTC, under grant numbers 113-2115-M-007-002-MY3 and 112-2115-M-032-004-MY3, for the first and second authors, respectively.

\end{document}